\documentclass[11pt,a4paper]{article}
\usepackage{fullpage}
\usepackage[utf8]{inputenc}
\usepackage{graphicx}
\usepackage{diagbox}
\usepackage{enumitem}
\usepackage{xcolor}
\usepackage{amsthm,amssymb, mathtools}
\usepackage{comment}
\usepackage{hyperref}
\hypersetup{
	colorlinks=true,       
	linkcolor=blue,        
	citecolor=magenta,         
	filecolor=magenta,     
	urlcolor=cyan,         
	linktoc=all
}
\usepackage[capitalise]{cleveref}
\usepackage{soul}
\usepackage{tikz}
\usetikzlibrary{shapes.misc}
\usepackage{booktabs}

\theoremstyle{plain}
\newtheorem{thm}{Theorem}
\newtheorem{lemma}[thm]{Lemma}
\newtheorem{cor}[thm]{Corollary}
\newtheorem{prop}[thm]{Proposition}
\newtheorem{conj}{Conjecture}

\theoremstyle{definition}

\newcounter{claim_nb}[thm]
\newtheorem{claim}[claim_nb]{Claim}
\newtheorem*{claim*}{Claim}

\newenvironment{claimproof}
{\begin{proof}
[Proof of Claim.]
\vspace{-1.2\parsep}}
{\renewcommand{\qed}{\hfill $\Diamond$} \end{proof}}

\DeclareMathOperator{\cuboid}{cuboid}

\DeclareMathOperator{\supp}{supp}
\DeclareMathOperator{\conv}{conv}
\DeclareMathOperator{\rank}{rank}
\DeclareMathOperator{\loc}{loc}
\DeclareMathOperator{\1}{{\bf 1}}
\DeclareMathOperator{\0}{{\bf 0}}
\DeclareMathOperator*{\argmin}{arg\,min}

\newcommand{\cB}{\mathcal{B}}
\newcommand{\cC}{\mathcal{C}}
\newcommand{\cD}{\mathcal{D}}
\newcommand{\bR}{\mathbb R}
\newcommand{\bZ}{\mathbb Z}

\renewcommand{\bar}[1]{\overline{#1}}

\DeclareMathOperator{\eye}{I}

\newif\ifnotes\notestrue 

\ifnotes
\usepackage{color}
\newcommand{\notename}[2]{{\textcolor{red}{\footnotesize{\bf (#1:} {#2}{\bf ) }}}}

\newcommand{\anote}[1]{{\notename{Ahmad}{#1}}}
\newcommand{\tnote}[1]{{\notename{Tamas}{#1}}}

\renewcommand{\r}[1]{{\color{red}{#1}}}

\else

\newcommand{\notename}[2]{{}}

\newcommand{\anote}[1]{}
\newcommand{\tnote}[1]{}

\renewcommand{\r}[1]{}

\fi

\title{Testing the max-flow min-cut property and the replication conjecture}

\author{Ahmad Abdi\thanks{Department of Mathematics, London School of Economics. Corresponding author, email address: a.abdi1@lse.ac.uk} \and
	Tam\'{a}s Schwarcz\thanks{Department of Mathematics, London School of Economics. Corresponding author, email address: t.b.schwarcz@lse.ac.uk} 
}

\begin{document}
	
\maketitle
	

\begin{abstract}
The replication conjecture [Conforti and Cornu\'{e}jols, 1993] states that every clutter with the packing property has the MFMC property. If true, this conjecture would have far-reaching consequences from integer programming and combinatorial optimization to commutative algebra. In this paper, we set out to verify the conjecture for the cuboid of a set-system in which the Hamming graph induced on the infeasible points has degree at most~$\delta$.

The family of cuboids of degree at most $\delta$ contains a rich source of clutters with the packing property, including all clutters over a ground set of size at most $\delta$. We prove that any minimal counterexample must have dimension at most $\delta$, thus making the target search space finite. We then use a state-of-the-art SAT solver to verify the replication conjecture for cuboids of degree at most $9$, and for clutters over at most $10$ elements.

Our computational verification relies crucially on another theoretical result, that to verify the MFMC property of a clutter over $n$ elements, it suffices to check finitely many weight vectors, namely $w\in \left\{0,1,\ldots,t\right\}^n$ where $t\leq\max\{\lceil n/2\rceil, \lfloor n-\sqrt{4n+1}+1\rfloor\}$. The upper bound of $t$ improves the previous best upper bound by algebraists, which could be exponential in $n$.\\


\end{abstract}
	
	\section{Introduction}
	
	A matrix $A\in \{0,1\}^{m\times n}$ has the \emph{max-flow min-cut (MFMC) property}, or \emph{is MFMC} for short, if the dual pair of linear programs \begin{align}
	&\min\{w^\top x: Ax\geq \1,x\geq \0\}\tag{P}\label{primal-set-covering}\\
	&\max\{\1^\top y:A^\top y\leq w, y\geq \0\}\tag{D}\label{dual-set-covering}
	\end{align} have integral optimal solutions for all $w\in \mathbb{Z}^n_{\geq 0}$. It is \emph{ideal} if \eqref{primal-set-covering} has an integral optimal solution for all $w\in \mathbb{Z}^n_{\geq 0}$. 
	In integer programming, MFMC matrices correspond to $0,1$ matrices $A$ whose set covering linear system $Ax\geq \1,x\geq \0$ is totally dual integral, and are closely related to ideal matrices. In combinatorial optimization and graph theory, the MFMC property is the theoretical underpinning for many min-max theorems such as the famous max-flow min-cut theorem, the Lucchesi--Younger theorem, Menger's theorem on packing edge-disjoint $st$-paths, and K\H{o}nig's theorem on maximum matchings in bipartite graphs~\cite{Cornuejols01}. In commutative algebra, it corresponds to when the symbolic and ordinary powers of square-free monomial ideals coincide~\cite{Ha10,Francisco13}.
	
	The most important open problem in the study of the MFMC property is the \emph{replication conjecture}, stating that if $A$ is not MFMC, then one of \eqref{primal-set-covering}, \eqref{dual-set-covering} will not have an integral optimal solution for some $w\in \{0,1,\infty\}^n$. If true, this would lead to a faster algorithm for testing the MFMC property, and have a knock-on impact towards a forbidden structure characterization for the class of MFMC matrices. It would also be the set covering analogue of Lov\'{a}sz's replication lemma for set packing systems, which was the crucial remaining piece in the proof of the famous perfect graph theorem~\cite{Lovasz72}.
	
	It will be helpful to work with a reformulation of the replication conjecture in terms of clutters. Let $\cC$ be a clutter over ground set $E$. Then $\cC$ has the \emph{packing property} if every minor $\cC'$ of it \emph{packs}, i.e., $\tau(\cC',\1)=\nu(\cC',\1)$; 
	it is an \emph{MFMC clutter} if $\tau(\cC,w)=\nu(\cC,w)$ for all $w\in \mathbb{Z}^E_{\geq 0}$. 
The replication conjecture equivalently states the following.
	
	\begin{conj}[replication conjecture \cite{Conforti93}]\label{replicon}
	A clutter with the packing property is MFMC.
	\end{conj}
	
	Perhaps the most compelling evidence for the replication conjecture is that clutters with the packing property are ideal~\cite{Cornuejols00}. This result is an immediate, indirect consequence of a powerful theorem by Lehman on the forbidden minors for ideal clutters~\cite{Lehman90}. As exciting as this result was at the time, there has been very little progress towards the conjecture ever since. 
	
	A major obstacle for proving the conjecture is our poor understanding of the packing property. For instance, as of writing this paper, there is no direct proof that the packing property implies idealness. This shortcoming becomes especially relevant when testing the replication conjecture: is there a class of clutters where the packing property is inherent but the MFMC is not?
	
	Another major obstacle is the remarkable gap between the two properties in the conjecture. While the packing property requires the existence of integral optimal solutions to \eqref{primal-set-covering}, \eqref{dual-set-covering} for only $3^n$ weight vectors, the MFMC property considers infinitely many weight vectors. How then would one even test in finite time whether a clutter has the MFMC property?
		
	In this paper, we introduce a framework that allows us to systematically test the replication conjecture, and along the way develop the theory further. 
	In what follows, we outline how the two issues discussed above are overcome in our work.
	
		\paragraph{Testing the MFMC property.} 
		Integer programmers would test the MFMC property by checking whether the inequality system $Ax\geq \1, x\geq \0$ of \eqref{primal-set-covering} is totally dual integral, or \emph{TDI} for short. This task boils down to testing, for every nonempty face of the feasible region, whether the corresponding tight rows of the coefficient matrix form a so-called \emph{Hilbert basis}. Indeed, this is a key idea for testing the TDI property in fixed dimension in polynomial time~\cite{Cook84}. It is also the key theoretical idea for testing the property in polyhedral geometry software packages such as \texttt{polymake}~\cite{polymake2000,polymake2017}.
		
		Given a clutter $\cC$ over $n$ elements, testing the MFMC property requires, by definition, checking the equality $\tau(\cC,w)=\nu(\cC,w)$ for infinitely many weight vectors $w\in \mathbb{Z}_{\geq 0}^n$. 
		Algebraists have shown that it suffices to check the equality for finitely many integral vectors $w$, namely those that satisfy $0\leq w_i\leq \big\lfloor (n+1)^{(n+3)/2}/2^n \big\rfloor$ for all $i\in [n]$~\cite[Theorem~5.6]{Herzog07}. 
		Using powerful positive characteristic techniques from commutative algebra, the upper bound was later improved to $\lceil m/2\rceil$, where $m$ is the number of sets in $\cC$~\cite[Theorem~5.2]{Montano21}. Though best possible in some cases, this bound gives a weak guarantee for problems in combinatorial optimization, where $m$ is typically exponential in $n$. 
		Here we drastically improve the dependency of the upper bound on $n$.
	
	\begin{thm}\label{mfmc-testing}
		Let $\cC$ be a clutter over $n$ elements. Then $\cC$ is MFMC if, and only if, $\tau(\cC,w)=\nu(\cC,w)$ 
		for all $w\in \bZ^n_{\geq 0}$ where every element belongs to some minimum $w$-weight cover of $\cC$, and $\tau(\cC,w)\leq \max\{\lceil n/2\rceil, \lfloor n-\sqrt{4n+1}+1\rfloor\}$.
	\end{thm}
	
	The bifurcated upper bound reflects the case analysis in the proof, which deals separately with ideal and non-ideal instances. Thanks to Lehman's theorem, we show fairly easily that for non-ideal instances, the upper bound $\lceil n/2\rceil$, and in fact $\lceil r/2\rceil$ for $r$ the rank of the incidence matrix of~$\cC$, suffices. 
	The crux of \Cref{mfmc-testing} is in testing the MFMC property for ideal clutters. By using Carath\'{e}odory's theorem, the rank-nullity theorem, and the idea of \emph{mates} in the context of forbidden minors for MFMC clutters, we prove the upper bound on $\tau(\cC,w)$ of $n-\sqrt{4n+1}+1$ for ideal clutters. While the upper bound is best possible in the non-ideal case, we believe that it is far from optimal in the ideal case, which we conjecture to be $2$!
	
	\paragraph{The $\tau=2$ conjecture.}
	\Cref{replicon} can be reformulated in terms of the excluded minors for the class of clutters with the packing property. 
	A clutter is \emph{minimally non-packing (mnp)} if it does not pack, but every proper minor does. 
	The replication conjecture equivalently states that no mnp clutter $\cC$ has \emph{replicates}, i.e., a pair $u,v$ of distinct elements such that if $A\in \cC$ and $A\cap \{u,v\}\neq \emptyset$, then $A\triangle \{u,v\}\in \cC$, for all $A\in \cC$. Here, $\triangle$ denotes the symmetric difference operator.
	
	The excluded minors for the packing property are separated into two classes: non-ideal mnp clutters, which are relatively well-understood and for which the replication conjecture reformulation holds, and ideal mnp clutters, which form the class that we focus on in this paper. Strikingly, all the known examples in this class have covering number $2$, thus preventing replicates from existing. 
	
	\begin{conj}[$\tau=2$ conjecture \cite{Cornuejols00}]\label{tau=2}
		Every ideal mnp clutter has covering number $2$.
	\end{conj}

	Given any ideal \emph{minimally non-MFMC} clutter $\cC$ over $n$ elements, with entry-wise minimal weights $w\in \bZ^{n}_{\geq 0}$ such that $\tau(\cC,w)>\nu(\cC,w)$, one obtains an ideal mnp clutter $\cC^w$ over $\1^\top w$ elements by replicating every element $u$, $w_u-1$ times (if $w_u=0$ then delete $u$)~\cite[Proposition~1]{Cornuejols00}. 
	An interesting question is to upper bound the number of elements of $\cC^w$ in terms of $n$. Of course, if the replication is true, then $w=\1$, so $\cC^w$ would have just $n$ elements. 
	While we cannot prove this bound, the proof of \Cref{mfmc-testing} does show that every entry of $w$ is upper bounded by $n-\sqrt{4n+1}+1$, so $\cC^w$ has at most $n^2-n\sqrt{4n+1}+n$ elements. We improve this upper bound by showing that $\1^\top w<(n-1)^2/4$, thus yielding the following result.
	
	\begin{thm}\label{tau=2->replicon}
		Let $n\geq 2$ be an integer. If the $\tau=2$ conjecture is true for all clutters with fewer than $(n-1)^2/4$ elements, then the replication conjecture holds for all clutters with at most $n$ elements.
	\end{thm}
	
	It should be emphasized that the $\tau=2$ conjecture is much stronger, and thus potentially easier to refute, than the replication conjecture. That said, it \emph{is} true for clutters over at most $8$ elements~\cite[Theorem~4.13]{abdi2020cuboids}.
	Thus, by \Cref{tau=2->replicon}, the replication conjecture is true for $n\leq 7$. In this work, we push beyond this limitation by improving the bound to $n\leq 10$.
	
	\paragraph{Cuboids of bounded degree.} 
	To describe our source of clutters where the packing property is inherent, we need to introduce an important object. 
	A \emph{cuboid} is a clutter $\cC$ whose ground set can be relabelled as $\{i,\bar{i}:i=1,\ldots,d\}$ such that $|C\cap \{i,\bar{i}\}|=1$ for all $C\in \cC, i\in [d]$. We refer to $d$ as the \emph{dimension} of the cuboid. Cuboids are most relevant when studying excluded minors for the packing property, in that all known ideal mnp clutters are either cuboids or their derivatives~\cite{ACP,abdi2020cuboids}.
	
	A key feature of cuboids is that their incidence matrix does not have full column rank, namely, it has rank at most $d+1$. This allows us to improve \Cref{mfmc-testing} for cuboids as follows. 
	
	\begin{thm}\label{mfmc-testing-cuboid}
		Let $\cC$ be a cuboid over ground set $E=\{i,\bar{i}:i=1,\ldots,d\}$ for some integer $d\geq 1$. Then $\cC$ is MFMC if, and only if, $\tau(\cC,w)=\nu(\cC,w)$ for all $w\in \bZ_{\geq 0}^{E}$ satisfying $w_{i}+w_{\bar{i}} = \tau(\cC,w)$ for all $i\in [d]$, where $\tau(\cC,w)\leq \max\{\lfloor 6(d+1)/7 \rfloor, \lfloor (1-1/\sqrt{2d})(d+1)\rfloor\}$. 
	\end{thm}
	
	This theorem gives a faster MFMC testing algorithm for cuboids compared to general clutters over the same number of elements. This is also a practical consideration as we set out to test the replication conjecture for cuboids.
		
	Another practical consideration is that cuboids over $2d$ elements can be described more compactly in a $d$-dimensional space, as follows. To each set $C$ in $\cC$, we associate a vector $p(C)\in \{0,1\}^d$ where $p(C)_i=1$ if and only if $C\cap \{i,\bar{i}\}=\{i\}$. In this manner, $\cC$ is the cuboid corresponding to the set-system $S\coloneqq\{p(C):C\in \cC\}\subseteq \{0,1\}^d$, denoted as $\cC=\cuboid(S)$.
	
	The set-system $S$ is \emph{cube-ideal} if $\cC$ is an ideal clutter; this property can also be described independently in terms of the set-system itself as we do in the next section. It is \emph{polar} if either $S\subseteq \{x:x_i=a\}$ for some $i\in [d]$ and $a\in \{0,1\}$, or $S$ contains antipodal points, i.e., a pair of points $p,q\in S$ such that $p+q=\1$. 
	It is \emph{strictly polar} if every restriction of $S$ obtained after fixing some coordinates to $0$ or $1$ is polar. For example, $\{000,110,101,011\}$ is not polar, $\{0000,1100,1010,0110\}$ is polar but is not strictly polar as it contains the former as a restriction.
	
	If $\cC$ has the packing property, then $S$ is a cube-ideal strictly polar set-system. The \emph{polarity conjecture} states that the converse is also true~\cite{abdi2020cuboids}. This conjecture is true if and only if the $\tau=2$ conjecture is true. In fact, the relationship between these two conjectures can be quantified as we see shortly.
	
	For an integer $\delta\geq 0$, the set-system $S$, or the cuboid $\cC$, has \emph{degree (at most) $\delta$} if the vertex-induced subgraph $G_d[\{0,1\}^d\setminus S]$ has maximum degree (at most) $\delta$, where $G_d$ denotes the skeleton graph of the hypercube $[0,1]^d$. For example, $\{00,10\}$ has degree $1$ while $\{000,110,101,011\}$ has degree $0$; see \Cref{fig:cuboid-degree-3} for another example. This terminology is due to Abdi, Cornu\'{e}jols, and Pashkovich who showed that there are only finitely many ideal mnp cuboids of bounded degree~\cite[Theorem~1.10~(i)]{ACP}.

\begin{figure}
\centering
\scalebox{0.7}{
\begin{tikzpicture}[
  o/.style={circle, draw, fill=black, minimum size=8pt, inner sep=0pt, line width=0.9pt},
  xm/.style={rectangle, draw, fill=white, minimum size=8pt, inner sep=0pt, line width=0.9pt},
  edge/.style={line width=1pt},
  hidden/.style={line width=1pt, dashed}
]

\begin{scope}
  \coordinate (FBL) at (0,0);     \coordinate (FBR) at (2.4,0);
  \coordinate (FTL) at (0,2.4);   \coordinate (FTR) at (2.4,2.4);
  \coordinate (BBL) at (0.9,0.9); \coordinate (BBR) at (3.3,0.9);
  \coordinate (BTL) at (0.9,3.3); \coordinate (BTR) at (3.3,3.3);
  \draw[edge] (FBL)--(FBR)--(FTR)--(FTL)--cycle;
  \draw[edge] (BBR)--(BTR)--(BTL);
  \draw[edge] (FBR)--(BBR);
  \draw[edge] (FTL)--(BTL); 
  \draw[edge] (FTR)--(BTR);
  \draw[hidden] (BBL)--(BBR);
  \draw[hidden] (BBL)--(BTL);
  \draw[hidden] (FBL)--(BBL);
  \node[o] at (FTL) {}; \node[o] at (FTR) {};
  \node[o] at (BTL) {}; \node[o] at (BTR) {}; \node[o] at (BBR) {};
  \node[xm] at (FBL) {}; \node[xm] at (FBR) {}; \node[xm] at (BBL) {};
\end{scope}

\begin{scope}[xshift=6cm]
  \coordinate (FBL) at (0,0);     \coordinate (FBR) at (2.4,0);
  \coordinate (FTL) at (0,2.4);   \coordinate (FTR) at (2.4,2.4);
  \coordinate (BBL) at (0.9,0.9); \coordinate (BBR) at (3.3,0.9);
  \coordinate (BTL) at (0.9,3.3); \coordinate (BTR) at (3.3,3.3);
  \draw[edge] (FBL)--(FBR)--(FTR)--(FTL)--cycle;
  \draw[edge] (BBR)--(BTR)--(BTL);
  \draw[edge] (FBR)--(BBR);
  \draw[edge] (FTL)--(BTL); 
  \draw[edge] (FTR)--(BTR);
  \draw[hidden] (BBL)--(BBR);
  \draw[hidden] (BBL)--(BTL);
  \draw[hidden] (FBL)--(BBL);
  \node[o] at (FTL) {}; \node[o] at (FTR) {}; \node[o] at (BTL) {};
  \node[o] at (BBL) {}; \node[o] at (BBR) {}; \node[o] at (FBR) {};
  \node[xm] at (BTR) {}; \node[xm] at (FBL) {};
\end{scope}

\end{tikzpicture}
}
\caption{A set-system $S\subseteq \{0,1\}^4$ of degree $3$, where the left and right cubes denote the restrictions $x_4=0$ and $x_4=1$, respectively. The round black vertices are precisely the points in $S$. The bottom left vertex has maximum degree $3$ in $G_4[\{0,1\}^4\setminus S]$.}
\label{fig:cuboid-degree-3}
\end{figure}
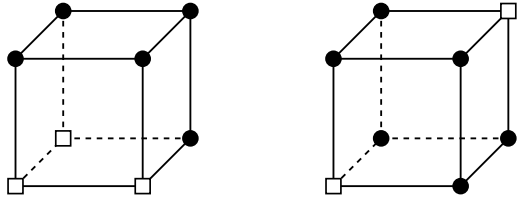
	
	The degree parameter is important to us because the polarity conjecture is true for set-systems of degree at most $\delta$ if and only if the $\tau=2$ conjecture holds for all clutters over at most $\delta$ elements. In particular, we have the following fact.
	
	\paragraph{Source of clutters with the packing property \cite[\S4]{abdi2020cuboids}.}
	\emph{Suppose the $\tau=2$ conjecture holds for all clutters over at most $\delta$ elements, for some integer $\delta\geq 0$. Let $d\geq 1$ be an arbitrary integer, and let $S\subseteq \{0,1\}^d$ be a cube-ideal strictly polar set-system of degree at most $\delta$. 
	Then $\cuboid(S)$ is guaranteed to have the packing property.}\bigskip
	
	The clutter $\cuboid(S)$ has the potential to be non-MFMC, thus being a counterexample to the replication conjecture. In fact, if this conjecture were false, then any counterexample $\cD$ over a minimum number $d$ of elements would lead to a counterexample of the above form, specifically for $S\subseteq \{0,1\}^d$ being the up-monotone set-system corresponding to $\cD$. 
	(Indeed, the replication conjecture holds for all cuboids if and only if it holds for all clutters.)
	
	
	Thus, for a fixed integer $\delta\geq 0$, we set out to test the replication conjecture for the cuboids of cube-ideal strictly polar set-systems of degree at most $\delta$. 	
	However, given that the class is infinite, can we even do the test in finite time? The following theorem answers this question in the affirmative, by allowing us to focus solely on cube-ideal strictly polar set-systems of dimension at most $\delta$.
		
	\begin{thm}\label{non-MFMC-cuboid}
	Every strictly polar set-system of degree $\delta$ with a non-MFMC cuboid has a restriction of dimension at most $\delta$ with a non-MFMC cuboid. In particular, if the replication conjecture holds for all cuboids of dimension at most $\delta$, then it holds for all cuboids of degree at most $\delta$.
	\end{thm}
	
We then use a computer program to prove the following theorem. 
	
\begin{thm} \label{cor:deg_nine}
	Let $d \ge 1$ be an arbitrary integer, and let $S\subseteq \{0,1\}^d$ be a cube-ideal strictly polar set-system. \begin{enumerate}[label=(\alph*)]
	\item If $S$ has degree at most $9$, then $\cuboid(S)$ has the MFMC property. 
	\item If $S$ is up-monotone, and $d\leq 10$, then $\cuboid(S)$ has the MFMC property. 
	\end{enumerate}
	In particular, the replication conjecture holds for cuboids of degree at most $9$, for cuboids over at most $18$ elements, and for clutters over at most $10$ elements. 
\end{thm}

Our computational approach proceeds irrespective of whether the $\tau=2$ conjecture holds for clutters over a certain number of elements, so the theorem above is not predicated on this conjecture. Of course, we have the added benefit that any counterexample to the replication conjecture would also yield a counterexample to the $\tau=2$ conjecture.

\paragraph{A SAT-based approach.} A natural first attempt at \cref{cor:deg_nine} would be to generate all cube-ideal strictly polar set-systems of small dimension and test the MFMC property for each.
The number of such set-systems grows rapidly, however, and this method is already infeasible in practice around dimension 6.
To reach dimension 9, we instead use a SAT-based approach that, to the best of our knowledge, is novel in the context of ideal clutters.
Using a refinement of \cref{mfmc-testing-cuboid}, for each fixed dimension $d$, we isolate a relatively small set of weight functions $w \in \bZ_{\ge 0}^{E}$ for which the equality $\tau(\cuboid(S), w)=\nu(\cuboid(S), w)$ needs to be verified.
For each such $w$, we encode the existence of a cube-ideal strictly polar set-system $S\subseteq \{0,1\}^d$ with $\nu(\cuboid(S), w) < \tau(\cuboid(S), w)$ as a SAT instance.
In this way, the computer-assisted part of the proof of \cref{cor:deg_nine} reduces to the certifying the unsatisfiability of a small number of SAT instances.
We establish the required unsatisfiability with the state-of-the-art SAT solver CaDiCaL~\cite{cadical}.
For each SAT instance, CaDiCaL produces a DRAT proof of unsatisfiability that can be checked independently with \texttt{drat-trim}~\cite{drattrim}.
For up-monotone set-systems, this approach lets us push the verification to dimension 10 as well.

\paragraph{Applications to commutative algebra.} 
For every clutter $\cC$ over ground set $[n]$, there is a canonical associated square-free monomial ideal $\eye\coloneqq \eye(\cC)$ over the polynomial ring $R\coloneqq \mathbb{C}[x_1,\ldots,x_n]$. More specifically, $\eye$ is the set of linear combinations of $\Pi_{e\in C} x_e,C\in \cC$, where the coefficients come from $R$. The square-free monomials $\Pi_{e\in C} x_e,C\in \cC$ are referred to as the \emph{generators} of the ideal, and $\eye$ is denoted as $\big(\Pi_{e\in C} x_e:C\in \cC\big)$. 

Strikingly, the MFMC property of $\cC$ can be equivalently characterized in commutative algebraic terms of the ideal $\eye$. In short, $\cC$ is MFMC if and only if the `ordinary' and `symbolic' powers of $\eye$ coincide throughout. Let us elaborate.

Let $k\geq 1$ be an integer. The $k\textsuperscript{th}$ \emph{ordinary power} $\eye^k$ refers to the ideal generated by all $k$-fold products of the generators of $\eye$. 

Symbolic powers in general have a rather technical definition, but in our setting where $\eye$ is equal to its radical and the host field $\mathbb{C}$ of the polynomial ring is algebraically closed of characteristic $0$, there is a much simpler and geometric definition that we can work with, thanks to the \emph{Zariski-Nagata theorem} for radical ideals~(see \cite[Corollary~2.9]{Sidman09}).

Denote by $V\coloneqq V(\eye)$ the subset of $\mathbb{C}^n$ on which every polynomial in $\eye$ vanishes.  
As $\eye$ is a radical ideal, this is reversible, so $\eye$ is the set of all polynomials that vanish on $V$. 

For an integer $k\geq 1$, the $k\textsuperscript{th}$ \emph{symbolic power} $\eye^{(k)}$ is the set of polynomials $f$ that vanish on $V$ up to `order $k$', i.e., $f$ and all its partial derivatives of order at most $k-1$ vanish on $V$. 

While $I^{k}\subseteq I^{(k)}$ for all $k\geq 1$, equality does not necessarily hold here. An illustrative example is $\eye=(xy,yz,zx)$. In this case, 
$\eye^2=(x^2y^2,y^2z^2,z^2x^2,x^2yz,xy^2z,xyz^2)$, which does not contain the polynomial $f(x,y,z)=xyz$. However, this polynomial does belong to $\eye^{(2)}$. To see this, note that $V(\eye) = \{(t,0,0):t\in \mathbb{C}\}\cup \{(0,t,0):t\in \mathbb{C}\}\cup \{(0,0,t):t\in \mathbb{C}\}$. As $xyz$ vanishes up to order $2$ on $V(\eye)$, it follows that $f\in \eye^{(2)}\setminus I^2$. 

Symbolic powers form a central topic in modern commutative algebra with core conjectures surrounding the comparison of ordinary and symbolic powers of ideals as initiated in \cite{Hochster02}. See the survey~\cite{Dao18} for more background on the state of the field. 
Integer programming turns out to be quite helpful, as for each integer $\tau\geq 1$, $\eye^k = \eye^{(k)}$ for all $k\in \{1,\ldots,\tau\}$ if and only if $\tau(\cC,w)=\nu(\cC,w)$ for all $w\in \bZ^{[n]}_{\geq 0}$ such that $\tau(\cC,w)\leq \tau$. In particular, $I^{k}= I^{(k)}$ for all $k\geq 1$ if and only if $\cC$ is MFMC (see \cite{Francisco13}). Subsequently, \Cref{mfmc-testing} can be rephrased equivalently as follows.

\begin{cor}\label{mfmc-testing-rephrased}
		Let $\eye$ be a square-free monomial ideal over $n$ variables. Then $\eye^{k}=\eye^{(k)}$ for all $k\geq 1$ if, and only if, $\eye^{k}=\eye^{(k)}$ for all $k\leq \max\{\lceil n/2\rceil, \lfloor n-\sqrt{4n+1}+1\rfloor\}$.\qed
	\end{cor}
	
	\Cref{mfmc-testing-cuboid} gives an improved upper bound on $k$ for cuboids, a class of clutters which are also studied in commutative algebra and referred to as `$2$-partitionable' clutters~\cite{FloresMendez08}.
	
	The replication conjecture is also an important problem in commutative algebra, though it is referred to as the `Conforti-Cornu\'{e}jols conjecture' or the `packing problem' in this community. \Cref{cor:deg_nine} verifies this conjecture for square-free monomials over at most $10$ variables, and those over at most $18$ variables that correspond to a $2$-partitionable clutter.

\paragraph{Outline.} After providing the omitted definitions and some preliminaries in \S\ref{section:prelim}, we prove \Cref{mfmc-testing}, \Cref{tau=2->replicon}, and \Cref{mfmc-testing-cuboid} in \S\ref{section:mfmc-testing} where we discuss MFMC property testing. \Cref{non-MFMC-cuboid} is proved in \S\ref{section:degree-dimension} where non-MFMC cuboids are studied. We explain our computer verification of \Cref{cor:deg_nine} in \S\ref{section:computational}, and finish off with two future research directions in \S\ref{section:conclusion}.
	
	\section{Preliminaries}\label{section:prelim}
	
	\paragraph{Clutters.} 
	A \emph{clutter} $\cC$ over a finite \emph{ground set} $E$ of \emph{elements} is a family of subsets of $E$ such that no set is contained in another. Its \emph{incidence matrix}, denoted $M(\cC)\eqqcolon A$, is the $\cC$-by-$E$ matrix whose rows are the incidence vectors of the sets in $\cC$. Given $w\in \bZ^E_{\geq 0}$, denote by 
	\begin{align}
	\tau(\cC,w):=&\min\left\{w^\top x: Ax\geq \1,x\geq \0, x\in \bZ^E\right\}\tag{IP}\label{I-primal-set-covering}\\
	\nu(\cC,w):=&\max\left\{\1^\top y:A^\top y\leq w, y\geq \0, y\in \bZ^{\cC}\right\},\tag{ID}\label{I-dual-set-covering}
	\end{align} both of which have finite optima. Denote by $\tau_\star(\cC,w),\nu_\star(\cC,w)$ the fractional analogues of the parameters above. Observe that $\tau(\cC,w)\geq \tau_\star(\cC,w)=\nu_\star(\cC,w)\geq \nu(\cC,w)$, where the equality follows from strong linear programming duality. 
	
	A \emph{cover} of $\cC$ is a subset of $E$ that intersects every set at least once; it is \emph{minimal} if it does not contain another cover. The family of minimal covers forms another clutter over the same ground set, which we call the \emph{blocker of $\cC$} and denote as $b(\cC)$. 
	The \emph{covering number} is the minimum size of a cover, which is also $\tau(\cC,\1)$. The \emph{packing number} is the maximum number of pairwise disjoint sets in $\cC$, which is also $\nu(\cC,\1)$. A clutter is \emph{intersecting} if $\tau(\cC,\1)\geq 2$ and $\nu(\cC,\1)=1$. 
	
	Given disjoint $I,J\subseteq E$, the \emph{minor} $\cC\setminus I/J$ obtained after \emph{deleting} $I$ and \emph{contracting} $J$ is the clutter over ground set $E\setminus (I\cup J)$ whose sets are the inclusion-wise minimal sets in $\{C\setminus J:C\in \cC, C\cap I=\emptyset\}$. The minor is \emph{proper} if $I\cup J\neq \emptyset$. 
	An element $e$ of $\cC$ is \emph{free} if it is contained in no member of $\cC$.
	To \emph{replicate} an element $e$ of $\cC$, we mean introducing a new copy $f$, and adding a set $C\triangle \{e,f\}$ to the clutter, for all sets $C\in \cC$ containing $e$. 
	
	Minor operations and replication can be emulated in terms of the optimization problems above. More specifically, deleting $e$ corresponds to setting $x_e=1$ and $w_e=0$, contracting $f$ corresponds to setting $x_f=0$ and $w_f=\infty$, and replicating $e$ corresponds to increasing the weight of $e$ by $1$. Subsequently, all these operations preserve both idealness and the MFMC property.
	
	\paragraph{Cuboids.} 
	Let $S\subseteq \{0,1\}^{[d]}$ be a set-system. 
	$S$ is \emph{up-monotone} if $p\leq q$ and $p\in S$ implies $q\in S$, for all $p,q\in \{0,1\}^{[d]}$. 	
	We let $[\bar{d}]\coloneqq \{\bar{1}, \dots, \bar{d}\}$.
	For disjoint subsets $I_0,I_1\subseteq [d]$, consider the set-system in $\{0,1\}^{[d]\setminus (I_0\cup I_1)}$ obtained from $S\cap \{x:x_i=0~\forall i\in I_0, x_j=1~\forall j\in I_1\}$ after dropping the coordinates in $I_0\cup I_1$; this is called a \emph{restriction of $S$}. 
	For $J\subseteq [d]$, the set-system in $\{0,1\}^{[d]\setminus J}$ after dropping the coordinates in $J$ from every point in $S$ is called a \emph{projection of $S$}. 
	A \emph{minor of $S$} is any set-system obtained after a series of restrictions and projections. 
	For $x\in \{0,1\}^{[d]}$, define the \emph{twisting} $S\triangle x:=\{p\triangle x:p\in S\}$, where $\triangle$ denotes entry-wise sum modulo $2$. 
	The \emph{localization of $S$ at $x$}, denoted $\loc(S;x)$, is the clutter over ground set $[d]$ whose sets correspond to the points in $S\triangle x$ of minimal support.
	For example, for $S=\{000,110,101,011\}$, the localization at $\0$ is $\{\emptyset\}$, while the localization at $\1$ is $\{\{1\},\{2\},\{3\}\}$.
	
	Recall that $S$ is polar if either $S\subseteq \{x:x_i=a\}$ for some $i\in [d]$ and $a\in \{0,1\}$, or $S$ contains antipodal points. Recall further that $S$ is strictly polar if every restriction of $S$ is polar. Equivalently, $S$ is polar if $\cuboid(S)$ packs, and $S$ is strictly polar if $\cuboid(S)$ has no intersecting minor. 
	
	Recall that $S$ is cube-ideal if $\cuboid(S)$ is an ideal clutter. Equivalently, $S$ is cube-ideal if its convex hull $\conv(S)$ can be described by $\0\leq x\leq \1$ and \emph{generalized set covering (GSC)} inequalities, which are of the form $\sum_{i\in I}x_i+\sum_{j\in J}(1-x_j)\geq 1$ for disjoint subsets $I,J\subseteq [d]$~\cite[Theorem~1.6]{abdi2020cuboids}. Also equivalent is that $S$ is cube-ideal if every localization is an ideal clutter~\cite[Theorem~1.8]{abdi2020cuboids}.
	
	Let $\mathfrak{P}$ be a minor-closed property defined on clutters. $\mathfrak{P}$ is a \emph{local property} if, for every set-system $S$, $\cuboid(S)$ has property $\mathfrak{P}$ if, and only if, all localizations of $S$ have property $\mathfrak{P}$. Thus, given the equivalence mentioned above, idealness is a local property. While the packing property is a non-local property, it is a local property for strictly polar set-systems~\cite[Theorem~1.11]{abdi2020cuboids}. This gives some evidence for the polarity conjecture, stating that for strictly polar set-systems, the cuboid is ideal if and only if it has the packing property.
	
	Let $\sum_{i\in I}x_i+\sum_{j\in J}(1-x_j)\geq 1$ and $\sum_{i\in I'}x_i+\sum_{j\in J'}(1-x_j)\geq 1$ be valid GSC inequalities for $S$, for disjoint subsets $I,J\subseteq [d]$ and disjoint subsets $I',J'\subseteq [d]$. Suppose $k\in I\cap J'$ for some index $k\in [d]$. It can be readily seen that $$
	\sum_{i\in I\cup I', i\neq k}x_i+\sum_{j\in J\cup J', j\neq k}(1-x_j)\geq 1$$ is also a valid inequality for $S$. This fact is known as the \emph{resolution principle} and the inequality is derived from the other two inequalities after \emph{resolving $k$}. Equivalently, given two covers $B_1,B_2$ of $\cuboid(S)$ such that $B_1\cap \{k,\bar{k}\}=\{k\}$ and $B_2\cap \{k,\bar{k}\}=\{\bar{k}\}$, the set $(B_1\cup B_2)\setminus \{k,\bar{k}\}$ is also a cover.
	
	\paragraph{Minimally non-ideal clutters.} A clutter is \emph{minimally non-ideal (mni)} if it is non-ideal but every proper minor is ideal. Given an integer $n\geq 3$, a \emph{delta of dimension $n$} is any clutter whose ground set can be relabelled as $[n]$ so that its sets are $\{1,2\},\{1,3\},\ldots,\{1,n\},\{2,3,\ldots,n\}$. Every delta is an intersecting mni clutter. We need the following powerful theorem.
	
	\begin{thm}[Lehman~\cite{Lehman90}]\label{lehman}
		Let $\cC$ be an mni clutter over ground set $[n]$. Then $$\tau(\cC,\1)>\tau_\star(\cC,\1)=\nu_\star(\cC,\1)>\nu(\cC,\1),$$ $\tau(\cC,\1)= \lceil\tau_\star(\cC,\1) \rceil$, $M(\cC)$ is a matrix of rank $n$, and $\tau_\star(\cC,\1)\leq \frac{n}{2}$. Furthermore, $\cC$ is either $\Delta_n,n\geq 4$, or $\cC$ has exactly $n$ minimum size sets of size $r\geq 2$, and exactly $n$ minimum covers of size $s\geq 2$, which can respectively be labelled as $C_1,\ldots,C_n\in \cC$ and $B_1,\ldots,B_n\in b(\cC)$ such that 
		$$
		|C_i\cap B_j| = \left\{
		\begin{array}{ll}
		1+\lambda& \quad\text{ if $i=j$}\\
		1& \quad\text{ if $i\neq j$}
		\end{array}
		\right.
		$$
		where $\lambda=rs-n\geq 1$. 
		\end{thm}
		
		In particular, an mni clutter does not pack, implying in turn that every clutter with the packing property is ideal. Subsequently, every mnp clutter is either ideal or mni.
	
	\paragraph{Non-MFMC clutters.} We need the following lemma, which borrows ideas from \cite[Proposition~3]{Cornuejols00}.

\begin{lemma}\label{minimal-non-MFMC-weights}
	Let $\cC$ be a clutter over ground set $E$. Suppose $\tau(\cC,w)>\nu(\cC,w)=:\nu$ for some $w\in \bZ_{\geq 0}^E$. 
	Choose $\bar{w}\in \bZ_{\geq 0}^E$ such that $\0\leq \bar{w}\leq w$, and $\bar{w}$ is entry-wise minimal subject to satisfying $\tau(\cC,\bar{w})>\nu(\cC,\bar{w})$. 
	Then the following statements hold: \begin{enumerate}[label=(\alph*)]
	\item \label{part:weights1} $\nu(\cC,\bar{w})\leq \nu$, 
	\item \label{part:weights2} every element appears in some minimum $\bar{w}$-weight cover, and $\tau(\cC,\bar{w})-\nu(\cC,\bar{w})=1$,
	\item \label{part:weights3} for every set $C\in \cC$ such that $\bar{w}\geq \1_C$, there exists a \emph{mate}, i.e., a minimal cover $B$ such that $\bar{w}(B)-|B\cap C|\leq \tau(\cC,\bar{w})-2$. In particular, $|B\cap C|\geq 2$.
	\end{enumerate}
\end{lemma}	
\begin{proof}
	{\bf \ref{part:weights1}} holds because 
	$\{y\in \bZ^{\cC}:A^\top y\leq \bar{w}, y\geq \0\}\subseteq \{y\in \bZ^{\cC}:A^\top y\leq w, y\geq \0\}$. 
	{\bf \ref{part:weights2}} If some element $i$ does not appear in any minimum $\bar{w}$-weight cover, then $\bar{w}_i>0$, and $$\tau(\cC,\bar{w}-e_i)=\tau(\cC,\bar{w})>\nu(\cC,\bar{w})\geq \nu(\cC,\bar{w}-e_i),$$ thus contradicting the minimal choice of $\bar{w}$; thus every element appears in some minimum $\bar{w}$-weight cover. Similarly, if $\tau(\cC,\bar{w})-\nu(\cC,\bar{w})\geq 2$, then given that $\tau(\cC,\bar{w}-e_i)\geq \tau(\cC,\bar{w})-1$ for any $i\in E$ with $\bar{w}_i>0$, it follows that $\tau(\cC,\bar{w}-e_i)>\nu(\cC,\bar{w}-e_i)$, which is again a contradiction; thus $\tau(\cC,\bar{w})=\nu(\cC,\bar{w})+ 1$. 
	{\bf \ref{part:weights3}} Suppose otherwise. Let $\tau:=\tau(\cC,\bar{w})$ and $w' \coloneqq \bar{w} - \1_C\geq \0$. Given that $C$ has no mate, it follows that $w'(B) = \bar{w}(B)-|B\cap C|\geq \tau-1$ for all $B\in b(\cC)$, so $\tau(\cC,w')\geq \tau-1$. By the minimal choice of $\bar{w}$, it follows that $\tau(\cC,w')=\nu(\cC,w')$. Let $y$ be an optimal solution for $\nu(\cC,w')$, which satisfies $\1^\top y=\nu(\cC,w')\geq \tau-1$. Then $y + \1_{\{C\}}$ is a feasible solution for $\nu(\cC,\bar{w})$ with objective value $\1^\top y+1$, so $\nu(\cC,\bar{w})\geq \1^\top y+1\geq \tau-1+1=\tau(\cC,\bar{w})$, a contradiction to the choice of $\bar{w}$.
	\end{proof}
	
	\paragraph{Non-MFMC cuboids.}
	
	We need the following result from \cite[Proposition~7.1]{abdi2021resistant}; there the hypotheses \ref{it:h2'} and \ref{it:h3'} are stronger than what we have written below, but the proof works all the same.
	
	\begin{lemma}[\cite{abdi2021resistant}] \label{prop:acl}
		Let $S\subseteq \{0,1\}^d$ be a polar set-system, and let $w\in \mathbb{Z}_{\ge 0}^{[d]\cup [\bar{d}]}$ be weights such that for $\cC=\cuboid(S)$, 
		\begin{enumerate}[label=(h\arabic*)]
			\item \label{it:h1'} $\tau\coloneqq\tau(\cC, w) > \nu(\cC, w)$,
			\item \label{it:h2'} $w$ is entry-wise minimal subject to satisfying \ref{it:h1'}, 
			\item \label{it:h3'} 
			for every single restriction $S'\subseteq \{0,1\}^{d-1}$ of $S$, and all weights $w'\in \mathbb{Z}^{2d-2}_{\ge 0}$ such that $\tau(\cC,w') = \tau$, we have $\tau(\cC,w')=\nu(\cC,w')$.
		\end{enumerate}
		Then the following statements hold:
		\begin{enumerate}[label=(\alph*)]
			\item \label{it:w} $w_{i} + w_{\bar{i}} = \tau$ and $1 \le w_i,w_{\bar{i}} \le \tau-1$ for each $i \in [d]$, and $\tau \ge 3$. In particular, $\cC$ has no cover of size one.
			\item \label{it:mate} Each $C \in \cC$ has a mate, i.e., a minimal cover $B$ such that $w(B) \le \tau-2+|B\cap C|$. Moreover, every mate $B$ is distinct from $\{i,\bar{i}\}, i\in [d]$, has at most two elements of weight at least $\frac{\tau}{2}$, and if it has two, say $f$ and $g$, then $B\subseteq C$, $w_f = w_g = \frac{\tau}{2}$ and $w_e = 1$ for each $e \in B\setminus \{f,g\}$. 
		\end{enumerate}
	\end{lemma}

	\paragraph{$2$-SAT.}
	An instance of the \emph{2-satisfiability problem} (\emph{2-SAT}) consists of a set of Boolean variables $X = \{x_1, \dots, x_n\}$ and a set of clauses $\mathcal{F}$ such that each clause $C\in \mathcal{F}$ is the disjunction of exactly two literals over $X$.
	The instance is \emph{satisfiable} if there exists an assignment $\phi\colon X \to \{\texttt{True}, \texttt{False}\}$ such that each clause contains at least one true literal.
	The \emph{implication graph} of the instance is the directed graph on vertex set $\{v_{x_1}, v_{\bar{x}_1}, \dots, v_{x_n}, v_{\bar{x}_n}\}$ with the following arcs: for each clause $(\ell_1\vee \ell_2)\in\mathcal{F}$, we add the arcs 	$v_{\bar\ell_1}v_{\ell_2}$ and $v_{\bar\ell_2}v_{\ell_1}$.
	Satisfiability of a 2-SAT instance can be characterized as follows.
	
	\begin{thm}[\cite{aspvall1979linear}] \label{prop:2sat} 
		A 2-SAT instance is unsatisfiable if and only if it has a variable $x_i$ such that the implication graph contains directed paths from $v_{x_i}$ to $v_{\bar{x_i}}$ and from $v_{\bar{x_i}}$ to $v_{x_i}$.
	\end{thm}
	
\section{Testing the MFMC property}\label{section:mfmc-testing}

Given a clutter $\cC$ over $n$ elements, testing the MFMC property requires, by definition, checking the equality $\tau(\cC,w)=\nu(\cC,w)$ for all weight vectors $w\in \bZ^n_{\geq 0}$. In this section, we prove \Cref{mfmc-testing}, showing that it suffices to check the equality for only finitely many vectors, by upper bounding the entries of $w$ by $n-\Omega(\sqrt{n})$. Our argument deals separately with the non-ideal and ideal instances. In fact, our upper bound guarantee is stronger in the non-ideal case thanks to Lehman's theorem, and is roughly $r/2$, where $r$ is the rank of the incidence matrix of $\cC$. In the ideal case, we give two upper bound guarantees, one depending solely on $n$, and another also depending on the rank $r$. In particular, for cuboids over $2d$ elements, we prove \Cref{mfmc-testing-cuboid}, showing an upper bound of $d-\Omega(\sqrt{d})$. Finally, we also prove \Cref{tau=2->replicon}, quantifying the relationship between the $\tau=2$ and replication conjectures.

\subsection{Testing the MFMC property for non-ideal clutters}

\begin{lemma}\label{idealness-testing-LE}
Let $\cC$ be a clutter over ground set $[n]$. If $\cC$ is non-ideal, then for any integer $t \geq \tau(\cC,w)=\lceil \tau_\star(\cC,w) \rceil$, there exists $w\in \{0,1,t\}^{[n]}$ such that \begin{enumerate}[label=(\alph*)]
\item $\tau(\cC,w)>\tau_\star(\cC,w)>\nu(\cC,w)$, and
\item $\frac{\rank(M(\cC))}{2}\geq \tau_\star(\cC,w)$.
\end{enumerate}
\end{lemma}
\begin{proof}
Let $\cC'\coloneqq\cC\setminus I/J$ be an mni minor for some disjoint $I,J\subseteq [n]$. By Lehman's \Cref{lehman}, we have 
$$\nu(\cC',\1)<\nu_\star(\cC',\1)=\tau_\star(\cC',\1)< \tau(\cC',\1),$$ where $\tau(\cC',\1)=\lceil \tau_\star(\cC',\1) \rceil$. 
Let $w\in \bZ_{\ge 0}^{[n]}$ be defined as $w_k=1$ for all $k\in [n]\setminus (I\cup J)$, $w_i=0~\forall i\in I$, and $w_j \geq \tau(\cC',\1) ~\forall j\in J$. We shall prove that this is the desired $w$. 

\begin{claim*} 
$\tau(\cC,w) = \tau(\cC',\1)$, $\nu_\star(\cC,w) = \nu_\star(\cC',\1)$, and $\nu(\cC,w)=\nu(\cC',\1)$.
\end{claim*}
\begin{claimproof}
$\tau(\cC,w)\leq \tau(\cC',\1)$ holds as any cover for $\cC'$ union $I$ gives a cover of the same weight for $\cC$. The reverse inequality $\tau(\cC,w)\geq \tau(\cC',\1)$ holds because any cover for $\cC$ is either disjoint from $J$ and hence gives a cover of the same weight for $\cC'$, or it intersects $J$ in which case it has weight at least $\tau(\cC',\1)$ as $w_j\geq\tau(\cC',\1)$ for all $j\in J$.

$\nu_\star(\cC,w)\geq \nu_\star(\cC',\1)$ because any solution to $\nu_\star(\cC',\1)$ naturally leads to a solution to $\nu_\star(\cC,w)$ of the same value; here we use crucially the inequality that $w_j\geq \nu_\star(\cC',\1)$ for all $j\in J$. The reverse inequality $\nu_\star(\cC,w)\leq \nu_\star(\cC',\1)$ holds because any solution to $\nu_\star(\cC,w)$ is comprised solely of sets in $\cC$ disjoint from $I$, so gives a solution to $\nu_\star(\cC',\1)$ of the same value.

The equality $\nu(\cC,w)=\nu(\cC',\1)$ holds for similar reasons as $\nu_\star(\cC,w)= \nu_\star(\cC',\1)$.
\end{claimproof}

Let $n'\coloneqq n-|I|-|J|$. We know from \Cref{lehman} that $M(\cC')$ is a matrix of rank $n'$, and $\tau_\star(\cC',\1) \leq \frac{n'}{2}$. Thus, $$
\tau_\star(\cC,w)=\tau_\star(\cC',\1) \leq  \frac{\rank(M(\cC'))}{2}\leq
\frac{\rank(M(\cC))}{2},$$ where the last inequality follows from the fact that $M(\cC')$ is a submatrix of $M(\cC)$. 
\end{proof}

Thus we obtain the following upper bound on the weighted covering numbers that need to be tested for checking the MFMC property for non-ideal clutters.

	\begin{thm}\label{mfmc-testing-non-ideal}
		Let $\cC$ be a clutter over $n$ elements. If $\cC$ is non-ideal, then there exists $w\in \bZ^n_{\geq 0}$ where every element belongs to some minimum $w$-weight cover of $\cC$, $\tau(\cC,w)-\nu(\cC,w)=1$, and $\tau(\cC,w)\leq \lceil \rank(M(\cC))/2\rceil$.
	\end{thm}
	\begin{proof}
	This is an immediate consequence of \Cref{idealness-testing-LE} and \Cref{minimal-non-MFMC-weights}.
	\end{proof}
	
	This upper bound is best possible due to the following extremal example $\cC$, for which $\rank(M(\cC))=n-1$.
	
	\begin{lemma}\label{extremal-example}
	Let $n\geq 4$ be an even integer, and let $$\cC
	\coloneqq\{\{i,i+1,n\}:i=1,\ldots,n-2\}\cup \{\{1,n-1,n\}\}.
	$$
	Then $\cC$ is non-ideal, and is therefore non-MFMC. Moreover, $\tau(\cC,w)=\tau_\star(\cC,w)=\nu(\cC,w)$ for all $w\in \bZ^{[n]}_{\geq \0}$ such that $w_n< \frac{n}{2}$. In particular, $\tau(\cC,w)=\nu(\cC,w)$ for all $w\in \bZ^{[n]}_{\geq \0}$ such that $\tau(\cC,w)< \frac{n}{2}$.
\end{lemma}
\begin{proof}
	Observe that $\cC/n$ is the clutter of edges of an odd cycle, which is non-ideal, so $\cC$ is non-ideal. It can be readily checked that $\cC\setminus i$ is MFMC for any $i\in [n-1]$. 
	Pick $w\in \bZ^{[n]}_{\geq 0}$ such that $w_n\leq \frac{n-2}{2}$. If $w_i=0$ for some $i\in [n-1]$, then $\tau(\cC,w)= \tau_\star(\cC,w)=\nu(\cC,w)$ since $\cC\setminus i$ is MFMC. Otherwise, $w_i\geq 1$ for all $i\in [n-1]$. Note that $\{n\}$ is a cover of $\cC$, so $w_n\geq\tau(\cC,w)$. We claim that $\nu(\cC,w)\geq w_n$, thereby proving the equality $\tau(\cC,w)=\nu(\cC,w)$. To this end, take the $\frac{n-2}{2}$ sets $C_i\coloneqq\{i,i+1,n\},i\in [n-2],i \text{ even}$, which are pairwise disjoint on $[n-1]$. Given that $w_i\geq 1,\forall i\in [n-1]$, and  $w_n\leq \frac{n-2}{2}$, the first $w_n$ sets among $C_2,C_4,\ldots,C_{n-2}$ guarantee that $\nu(\cC,w)\geq w_n$, as required.
	\end{proof}

\subsection{Testing the MFMC property for ideal clutters}

\begin{thm}\label{mfmc-testing-ideal}
	Let $\cC$ be an ideal clutter over $n$ elements. If $\cC$ is not MFMC, then there exists $w\in \bZ_{\geq 0}^n$ where every element belongs to some minimum $w$-weight cover of $\cC$, $\tau(\cC,w)-\nu(\cC,w)=1$, and \begin{enumerate}[label=(\alph*)]
	\item\label{mfmc-testing-ideal:average} $\1^\top w< \left(\frac{n-1}{2}\right)^2$,
	\item\label{mfmc-testing-ideal:tau-n} $\tau(\cC,w)< n-2\sqrt{n}+1$, and
	\item\label{mfmc-testing-ideal:tau-m-r} $\tau(\cC,w)< \left(1-\frac{1}{\sqrt{2m}}\right)r$, where $r=\rank(M(\cC))$ and $m=|\{i\in [n]:w_i\leq 3(r-\tau)\}|$.
	\end{enumerate}
\end{thm}
\begin{proof}
	Let $E$ be the ground set, $M\coloneqq M(\cC)$ and $\cB\coloneqq b(\cC)$. As $\cC$ does not have the MFMC property, 
$\tau(\cC,w)>\nu(\cC,w)$ for some $w\in \bZ_{\geq 0}^E$. 
	Choose $\hat{w}\in \bZ_{\geq 0}^E$ such that $\0\leq \hat{w}\leq w$, and $\hat{w}$ is entry-wise minimal subject to satisfying $\tau(\cC,\hat{w})>\nu(\cC,\hat{w})$. 
	By \Cref{minimal-non-MFMC-weights}, $\tau(\cC,\hat{w})-\nu(\cC,\hat{w})=1$, and every element appears in some minimum $\hat{w}$-weight cover. 
	It follows from idealness that $\tau(\cC,\hat{w})=\nu_\star(\cC,\hat{w})$. Let $\tau\coloneqq \tau(\cC,\hat{w})$ and $\nu\coloneqq \nu(\cC,\hat{w})$. 
	By deleting all weight-$0$ elements, if necessary, we may assume that $\hat{w}\geq \1$. Note that this operation keeps the clutter non-MFMC, and does not change the left-hand sides and decreases the right-hand sides of the inequalities we plan to prove. 
	
	\begin{claim}\label{mfmc-testing-strong:mate} 
For each $C\in \cC$, there exists a $B\in \cB$ such that $\hat{w}(B)-|B\cap C|\leq \tau-2$; we call $B$ a \emph{mate} of $C$.
\end{claim}
\begin{claimproof}
This follows from \Cref{minimal-non-MFMC-weights} after noting that $\hat{w}\geq \1$.
\end{claimproof}
	
	\begin{claim}\label{mfmc-testing-strong:fp-up} 
Let $\hat{y}$ be an optimal solution for $\nu_\star(\cC,\hat{w})$. Then $
\hat{y}_C\leq 1 - \frac{1}{|C\cap B|-1}$ for each $C\in \cC$ and each mate $B$ of $C$. In particular, $\hat{y}\in [0,1)^{\cC}$. 
\end{claim}
\begin{claimproof}
Fix $C\in \cC$. Let $$
w'\coloneqq \hat{w} - \hat{y}_C \cdot \1_C = \sum_{C'\in \cC, C'\neq C} \hat{y}_{C'} \1_{C'},
$$ where the right-hand side is a fractional $w'$-packing of value $\tau-\hat{y}_C$. Thus, for every $B\in \cB$, $$
\hat{w}(B) - \hat{y}_C |B\cap C|=w'(B)\geq \tau-\hat{y}_C,
$$ implying in turn that $$
\hat{y}_C\leq \frac{\hat{w}(B)-\tau}{|C\cap B|-1} \quad\forall C\in \cC, B\in \cB,
$$
where $\frac{0}{0}=:1$. Now pick $B$ to be the mate of $C$ per \Cref{mfmc-testing-strong:mate}, which by definition satisfies $\hat{w}(B) - |C\cap B|\leq \tau-2$. 
In particular, $|C\cap B|\geq 2$, and $$
\hat{y}_C\leq \frac{\hat{w}(B)-\tau}{|C\cap B|-1}\leq 
1 - \frac{1}{\hat{w}(B)-\tau+1}
\leq
1 - \frac{1}{|C\cap B|-1},
$$ as required.
\end{claimproof}

Let $\cC'$ be the clutter of sets $C$ in $\cC$ for which there is an optimal solution $\hat{y}$ for $\nu_\star(\cC,\hat{w})$ such that $\hat{y}_C>0$, and let $M'\coloneqq M(\cC')$.  

\begin{claim}\label{mfmc-testing-strong:CS-supp} 
The following statements hold:
\begin{enumerate}[label=(\alph*)]
\item \label{part:CS-supp1}
$\hat{y}\in \bR^{\cC}$ is an optimal solution for $\nu_\star(\cC,\hat{w})$ if and only if $M'^\top \hat{y} = \hat{w}, \hat{y}\geq \0$.
\item \label{part:CS-supp2}
There is an optimal solution $\hat{y}\in \bR^{\cC}$ for $\nu_\star(\cC,\hat{w})$ such that $|\supp(\hat{y})|\leq \rank(M')$. 
\end{enumerate}
\end{claim}
\begin{claimproof}
{\bf \ref{part:CS-supp1}}
Recall that every element belongs to some minimum $\hat{w}$-weight cover of $\cC$. 
As $\cC$ is ideal, it follows that every element in $[n]$ appears in the support of an optimal solution to $\tau_\star(\cC,\hat{w})$, so this part follows from complementary slackness. 
{\bf \ref{part:CS-supp2}} follows from \ref{part:CS-supp1} and Carath\'{e}odory's theorem.
\end{claimproof}

We first upper bound $\tau$ in terms of the rank of $M'$. To this end, let $\cB'$ be the clutter of minimum $\hat{w}$-weight covers of $\cC$, and let $N'\coloneqq M(\cB')$. By complementary slackness and strict complementarity, a set $C$ in $\cC$ belongs to $\cC'$ if, and only if, $C$ intersects every set in $\cB'$ exactly once. In particular, $M'N'^\top = J$, where $J$ is the all-ones matrix of appropriate dimensions.

\begin{claim}\label{mfmc-testing-strong:tau-supp-ub} 
Let $\hat{y}$ be an optimal solution for $\nu_\star(\cC,\hat{w})$. Then $\tau\leq |\supp(\hat{y})| - \max_{B\in \cB'} |B|$. Subsequently, $\tau\leq \rank(M')-\max_{B\in \cB'} |B|$.
\end{claim}
\begin{claimproof}
The second part follows from the first part and \Cref{mfmc-testing-strong:CS-supp}, part~\ref{part:CS-supp2}. Thus, it suffices to prove the first part. To this end, fix $B\in \cB'$. 
For each $e\in B$, $$\hat{w}_e = \sum_{C\in \cC,e\in C} \hat{y}_C= \sum_{C\in \cC',e\in C} \hat{y}_C,$$ where the first and second equalities follow from \Cref{mfmc-testing-strong:CS-supp}, part~\ref{part:CS-supp1} and the definition of $\cC'$, respectively. Thus, by \Cref{mfmc-testing-strong:fp-up}, $|\{C\in \cC':e\in C\}|\geq \hat{w}_e+1$. As $|B\cap C|=1$ for all $C\in \cC'$, we have $$
 |\cC'| = \sum_{e\in B} |\{C\in \cC':e\in C\}|\geq \hat{w}(B) + |B|,
$$ thus proving the claim.
\end{claimproof}

Next we give an upper bound on the rank of $M'$. 

\begin{claim}\label{mfmc-testing-strong:rank-ub} 
Let $t$ be the minimum number of sets in $\cB'$ whose union is the ground set.
Then $$
n+1- \frac{n}{\max_{B\in \cB'} |B|}\geq n+1-t\geq \rank(M').
$$ Furthermore, if equality holds throughout, then $[n]$ can be partitioned into sets in $\cB'$ of maximum size.
\end{claim}
\begin{claimproof}
As $M'N'^\top = J$, it follows from the rank-nullity theorem that $\rank(M')+\rank(N')\leq n+1$.
Given that every element in $[n]$ appears in a set of $\cB'$, it follows that $\rank(N')\geq t$. Clearly, $t\geq n/\max_{B\in \cB'} |B|$. Putting all these together we obtain the claimed inequality. If equality holds, then $t= n/\max_{B\in \cB'} |B|$, implying that the minimum covering consists of $t$ pairwise disjoint sets in $\cB'$ of maximum size, as claimed.
\end{claimproof}

We are ready to provide an upper bound on $\1^\top \hat{w}$ in terms of $n$. 

\begin{claim}
$\left(\frac{n-1}{2}\right)^2> \1^\top \hat{w}$. 
\end{claim}
\begin{claimproof}
By \Cref{mfmc-testing-strong:CS-supp}, part~\ref{part:CS-supp2}, we may choose an optimal solution $\hat{y}\in \bR^{\cC}$ for $\nu_\star(\cC,\hat{w})$ to have support size at most $\rank(M')$. By \Cref{mfmc-testing-strong:CS-supp}, part~\ref{part:CS-supp1}, we have $$
\1^\top \hat{w} = \sum_{e\in [n]} \sum_{C\in \cC',e\in C} \hat{y}_C = \sum_{C\in \cC'} \hat{y}_C |C|\leq \tau \cdot \max_{C\in \cC'} |C|.
$$
Let $t$ be the minimum number of sets in $\cB'$ whose union is the ground set. 
Given that every set in $\cC'$ intersects every set of the covering exactly once, it follows that $|C|\leq t$ for all $C\in \cC'$. Furthermore, 
$\tau\leq |\supp(\hat{y})|-\max_{B\in \cB'} |B|$ by \Cref{mfmc-testing-strong:tau-supp-ub}, $|\supp(\hat{y})|\leq \rank(M')$ by the choice of $\hat{y}$, and 
$\rank(M')\leq n+1-t$ by \Cref{mfmc-testing-strong:rank-ub}, so 
$$
\1^\top \hat{w} 
\leq (n+1-t-\max_{B\in \cB'} |B|)\cdot t \leq \left(\frac{n+1-\max_{B\in \cB'} |B|}{2}\right)^2
$$ where the right-most inequality follows from the arithmetic mean--geometric mean inequality. Given that $\cC$ is not MFMC, it follows that $\max_{B\in \cB'}|B|\geq 2$, so $\1^\top \hat{w}\leq \left(\frac{n-1}{2}\right)^2$. If equality holds, then $\max_{B\in \cB'}|B| = 2$ and $t = (n-1)/2$, which is not possible as $[n]$ could not be covered with $(n-1)/2$ sets of size at most $2$. The claim follows.
\end{claimproof}

Next we upper bound $\tau$ in terms of $n$. 

\begin{claim}
$n-2\sqrt{n}+1> \tau$. 
\end{claim}
\begin{claimproof}
By \Cref{mfmc-testing-strong:CS-supp}, part~\ref{part:CS-supp2}, we may choose an optimal solution $\hat{y}\in \bR^{\cC}$ for $\nu_\star(\cC,\hat{w})$ to have support size at most $\rank(M')$. Thus, $$
n +1 -  \frac{n}{\max_{B\in \cB'} |B|} - \max_{B\in \cB'} |B|\geq \rank(M')- \max_{B\in \cB'} |B| \geq |\supp(\hat{y})|-\max_{B\in \cB'} |B|\geq \tau
$$ where 
the left- and right-most inequalities follow from \Cref{mfmc-testing-strong:rank-ub} and \Cref{mfmc-testing-strong:tau-supp-ub}, respectively. Given that $\argmin_{x>0} x+\frac{n}{x} = \sqrt{n}$ for any $n>0$, it follows that $$
n-2\sqrt{n}+1\geq \tau.
$$
It remains to show that equality cannot hold here. Suppose otherwise. Then for some integer $r\geq 1$, we have $n = r^2$, $\tau=(r-1)^2$, $\max_{B\in \cB'} |B|=\sqrt{n}=r$, and $$|\supp(\hat{y})|=\rank(M')=r^2+1-\frac{n}{r}=r^2+1-r.$$ It follows from \Cref{mfmc-testing-strong:rank-ub} that $[n]$ can be partitioned into $r$ sets in $\cB'$ of maximum size. This in turn implies that every set in $\cC'$ has size $r$. By \Cref{mfmc-testing-strong:fp-up}, we have 
$$
(r-1)^2=\tau = \1^\top \hat{y}\leq \sum_{C\in \cC',\hat{y}_C>0}\left(1-\frac{1}{|C|-1}\right) =
|\supp(\hat{y})|\left(1-\frac{1}{r-1}\right) = (r-1)^2 - \frac{1}{r-1},
$$ a contradiction. 
\end{claimproof}

Finally, we prove an alternate upper bound on $\tau$, namely the inequality in part~\ref{mfmc-testing-ideal:tau-m-r}. 
	As $\cC$ is non-MFMC, some minimum $\hat{w}$-weight cover must have size at least two. In fact, by contracting all elements $e\in E$ where $\{e\}$ is a minimum $\hat{w}$-weight cover, if necessary, we may assume that each minimum $\hat{w}$-weight cover has size at least two. Note that this operation keeps the clutter non-MFMC, does not change $\tau$, and can only decrease the right-hand side of the inequality that we plan to prove.
	
	By \Cref{mfmc-testing-strong:CS-supp}, part~\ref{part:CS-supp2},  we may choose an optimal solution $\hat{y}$ of $\nu_\star(\cC, \hat{w})$ where 
	$$s \coloneqq |\supp(\hat{y})|\le \rank(M') \eqqcolon r'.$$ By \Cref{mfmc-testing-strong:fp-up}, $s>\tau$. 
	We need the following inequality.
	
	\begin{claim} \label{cl:average-mate-intersection}
		For each $C \in \supp(\hat{y})$, let $B_C$ be a mate of $C$ provided by \cref{mfmc-testing-strong:mate}. Then $$\frac{1}{s}\sum_{C \in \supp(\hat{y})} |B_C\cap C| \ge \frac{s}{s-\tau} + 1.$$
	\end{claim}
	\begin{claimproof}
		Using \cref{mfmc-testing-strong:fp-up} and the arithmetic mean--harmonic mean inequality,
		\begin{align*}
			\tau & = \sum_{C \in \cC} \hat{y}_C \le \sum_{C \in \supp(\hat{y})} \left(1-\frac{1}{|B_C\cap C|-1}\right)  = s - \sum_{C \in \supp(\hat{y})} \frac{1}{|B_C \cap C|-1} \\
			& \le s - \frac{s^2}{\sum_{C \in \supp(\hat{y})} (|B_C \cap C|-1)} = s - \frac{s^2}{- s + \sum_{C \in \supp(\hat{y})} |B_C \cap C|},
		\end{align*}
		and the claim follows after rearranging.
	\end{claimproof}
	
	We use this claim to show that on average, the sets 
	$C\in \supp(\hat{y})$ 
	contain a large number of elements of ``small" weight. 
	More precisely, fix an integer $k \in \{1,\ldots,\tau-1\}$ that we will use as the threshold for a weight being small. 
	For $i \in [k]$, let $E_i \coloneqq \{e \in E: \hat{w}_e = i\}$ and $E_{\le i} \coloneqq E_1 \cup \dots \cup E_i$.
	Consider the following key quantity: 
	$$\Phi_k : = \frac{1}{sk}\sum_{C\in \supp(\hat{y})} \sum_{i=1}^k |C \cap E_{\leq i}|,$$
	i.e., the average value of $\frac{1}{k}\sum_{i=1}^k |C \cap E_{\leq i}|$ for $C \in \supp(\hat{y})$. 
	Note that $\frac{1}{k}\sum_{i=1}^k |C \cap E_{\leq i}|\leq |C\cap E_{\leq k}|$, the number of elements in $C$ of weight at most $k$. 
	 We shall sandwich $\Phi_k$ between two functions of $n,s,\tau$, and $k$.
	
	\begin{claim} \label{cl:average-mate-small-weight}
		$\Phi_k
		\ge \frac{s}{s-\tau}-\frac{\tau-2}{k}+1$.
	\end{claim}
	\begin{claimproof}
	For each $C \in \supp(\hat{y})$, let $B_C$ be a mate of $C$. Then
	$$
	\Phi_k\geq \frac{1}{sk} \sum_{C \in \supp(\hat{y})} \sum_{i=1}^k |B_C \cap C \cap E_{\leq i}|  = \frac{1}{s} \sum_{C \in \supp(\hat{y})} \sum_{i=1}^k \frac{k+1-i}{k} \cdot |B_C \cap C \cap E_i|.
	$$ We prove the desired inequality with $\Phi_k$ replaced by the right-hand side value above. For each $C \in \cC$, using $\hat{w}\ge \1$ and the definition of mates, we have  
		\begin{align*}
			|B_C \cap C| + \tau - 2 & \ge \hat{w}(B_C) \ge \hat{w}(B_C \cap C) \\ & \ge \sum_{i=1}^k |B_C \cap C \cap E_i| \cdot i + \left(|B_C \cap C | - \sum_{i=1}^k |B_C \cap C \cap E_i|\right) \cdot (k+1),
		\end{align*}
		hence 
		\begin{align*}
			\frac{\tau-2 + \sum_{i=1}^k (k+1-i) |B_C \cap C \cap E_i|}{k} \ge |B_C \cap C|.
		\end{align*}
		Summing up this inequality for all $C\in \supp(\hat{y})$, dividing by $s$, and then applying \cref{cl:average-mate-intersection}, yields the claimed inequality.
	\end{claimproof}
	
	On the other hand, the equality $M'^\top \hat{y} = \hat{w}$ implies an upper bound on $\Phi_k$.
	
	\begin{claim} \label{cl:average-set-small-weight}
		$\Phi_k
		\leq \frac{1}{ks} \sum_{i=1}^k (k+1-i)(s-\tau + i - 1) |E_i|
		\leq \frac{(k+s-\tau)^2}{4ks}\cdot |E_{\leq k}|$.
	\end{claim}
	\begin{claimproof}
	Observe that 
	$$\Phi_k=\frac{1}{s} \sum_{C \in \supp(\hat{y})} \sum_{i=1}^k \frac{k+1-i}{k} |C\cap E_i|.$$
	We prove the desired inequality with $\Phi_k$ replaced by the right-hand side value above.
	Fix $e\in E$, and let $B$ be a minimum $\hat{w}$-weight cover containing $e$, which has size at least two. Then
		\begin{align*}
			|\{C \in \supp(\hat{y}): e \in C\}| & = |\supp(\hat{y})|-\sum_{f \in B-e} |\{C \in \supp(\hat{y}): f \in C\}| \\
		& \le s - \sum_{f \in B-e} (\hat{w}_f+1) = s - (\tau - \hat{w}_e) - (|B|-1) \le s-\tau + \hat{w}_e - 1,
		\end{align*} where the first (in)equality follows from the fact that $B$ intersects every set in $\supp(\hat{y})$ exactly once, and the second one follows from \Cref{mfmc-testing-strong:fp-up}. 	
		Subsequently,
		\begin{align*}
			\sum_{C \in \supp(\hat{y})} \sum_{i=1}^k \frac{k+1-i}{k} \cdot |C \cap E_i| &= \sum_{i=1}^k \frac{k+1-i}{k} \sum_{e \in E_i} |\{C \in \supp(\hat{y}): e \in C\}| \\
			& \le \sum_{i=1}^k \frac{k+1-i}{k} \sum_{e \in E_i} (s-\tau + i - 1) \\
			& = \frac{1}{k} \sum_{i=1}^k (k+1-i)(s-\tau + i - 1) |E_i| \\
			& \le \frac{1}{k} \sum_{i=1}^k \left(\frac{k+s-\tau}{2}\right)^2 |E_i| \\
			& = \frac{(k+s-\tau)^2}{4k} \cdot |E_{\leq k}|,
		\end{align*} where the second to last inequality follows from the arithmetic mean--geometric mean inequality. 
		Dividing by $s$ proves the claim.
	\end{claimproof}
	
	Let $d \coloneqq s-\tau\geq 1$, $k
	\coloneqq 3d$, and $m'\coloneqq |E_{\leq k}|$. Then by \Cref{cl:average-mate-small-weight} and \Cref{cl:average-set-small-weight} we have
	\[
		\frac{s}{d} - \frac{s-d-2}{3d} + 1 \leq \Phi_k\leq \frac{(4d)^2}{12ds}\cdot m'.
	\]
	Multiplying both sides by $\frac{3ds}{2}$ yields
	$s^2  + 2ds + s \le 2d^2 m'$, 
	which after ignoring $2ds+s$ from the left-hand side implies that $\frac{s}{\sqrt{2m'}} < d$. 
	In terms of $\tau$, this yields
	\[\tau = s-d < \left(1-\frac{1}{\sqrt{2m'}}\right) s \le \left(1-\frac{1}{\sqrt{2m'}}\right) r'.\]
	Given that $r'\leq \rank(M(\cC))=r$, $k\leq 3(r'-\tau)\leq 3(r-\tau)$, and $m'\leq |E_{\leq 3(r-\tau)}|=m$, the inequality in part~\ref{mfmc-testing-ideal:tau-m-r} follows.
\end{proof}

\subsection{Bringing it altogether}

\begin{proof}[Proof of \Cref{mfmc-testing}]
\Cref{mfmc-testing} follows from \Cref{mfmc-testing-non-ideal}, and \Cref{mfmc-testing-ideal}, part~\ref{mfmc-testing-ideal:tau-n} after noting that no integer lies strictly between $n-2\sqrt{n}+1$ and $n-\sqrt{4n+1}+1$.
\end{proof}

\begin{proof}[Proof of \Cref{mfmc-testing-cuboid}]
	The forward implication holds by definition, so let us focus on its contrapositive. 
	Suppose $\cC$ is not MFMC. Then by \Cref{mfmc-testing-non-ideal} and \Cref{mfmc-testing-ideal}, there exists $\bar{w}\in \bZ_{\geq 0}^E$ such that $\tau(\cC,\bar{w})-\nu(\cC,\bar{w})=1$, every element appears in some minimum $\bar{w}$-weight cover, amongst other criteria that will discuss shortly. Let $\tau\coloneqq \tau(\cC,\bar{w})$.
	
	We claim that each $\{i,\bar{i}\}$ is a minimum $\bar{w}$-weight cover. 
	Suppose otherwise. Let $B,B'$ be minimum $\bar{w}$-weight covers containing $i,\bar{i}$, respectively. Let $B''\coloneqq(B\cup B')\setminus \{i,\bar{i}\}$, which is also a cover of $\cC$ by the resolution principle. Note however that $$
	w(B'') = w(B)+w(B') - w_i-w_{\bar{i}} - w(B\cap B')<\tau+\tau-\tau = \tau,
	$$ a contradiction. 
	
	Let $M\coloneqq M(\cC)$. As $\cC$ is a cuboid, it follows that $r\coloneqq \rank(M)\leq d+1$. Thus, if $\cC$ is non-ideal, then \Cref{mfmc-testing-non-ideal} implies that $\tau\leq \lceil (d+1)/2 \rceil$. 
	Otherwise, $\cC$ is ideal, and by \Cref{mfmc-testing-ideal}, part~\ref{mfmc-testing-ideal:tau-m-r}, $\tau< (1-1/\sqrt{2m}) (d+1)$, where $m$ is the number of elements of $\bar{w}$-weight at most $3(r-\tau)$. If $r<7\tau/6$, then $3(r-\tau)< \tau/2$, so $m\leq d$, implying in turn that $\tau< (1-1/\sqrt{2d}) (d+1)$.
	Otherwise, $d+1\geq r\geq 7\tau/6$, so $\tau\leq 6(d+1)/7$, thus finishing the proof.
\end{proof}

\begin{proof}[Proof of \Cref{tau=2->replicon}]
	Suppose the $\tau=2$ conjecture is true for all clutters with fewer than $(n-1)^2/4$ elements. Let $\cC$ be a clutter with the packing property over $n$ elements. Then $\cC$ is ideal by Lehman's \Cref{lehman}. Let $w\in \bZ^n_{\geq 0}$ such that $\1^\top w< (n-1)^2/4$. To prove that $\cC$ is MFMC, it suffices to show that $\tau(\cC,w)=\nu(\cC,w)$ by \Cref{mfmc-testing-ideal}, part~\ref{mfmc-testing-ideal:average}. Let $\cC^w$ be the clutter obtained from $\cC$ after replicating each element $i$, $w_i-1$ times (if $w_i=0$, then delete $i$). As $\cC$ is ideal, so is $\cC^w$. We need to show that $\cC^w$ packs. Suppose otherwise. Then $\cC^w$ has an ideal mnp minor $\cC'$. As $\cC^w$ has $\1^\top w< (n-1)^2/4$ elements, the $\tau=2$ conjecture holds for $\cC'$, so $\tau(\cC')=2$. As $\cC$ has the packing property, $\cC'$ is not its minor, so $\cC'$ must contain a pair of replicates, say $i\neq j$. As $\cC'$ is mnp, $i$ belongs to some minimum cover $B$. As $i,j$ are replicates, it follows that $j\in B$, so $B=\{i,j\}$. Given that $\cC'/i $ packs, and $\tau(\cC'/i)\geq \tau(\cC')=2$, $\cC'/i= \cC'/i\setminus j$ contains disjoint sets $C_1,C_2$. But then $C_1,C_2\triangle \{i,j\}$ are disjoint sets of $\cC'$, a contradiction as $\cC'$ is non-packing.
	\end{proof}
	
	\section{From bounded degree to bounded dimension}\label{section:degree-dimension}

In this section, we prove that if the replication conjecture holds for all cuboids of dimension at most $\delta$, then it holds for all cuboids of degree at most $\delta$. This is done by proving \Cref{non-MFMC-cuboid}, showing that every strictly polar set-system of degree $\delta$ with a non-MFMC cuboid has a restriction of dimension at most $\delta$ with a non-MFMC cuboid. To this end, we need a few definitions.
	
	Let $d \ge 1$ be an integer, and let $S\subseteq \{0,1\}^d$ be a set-system. We refer to the points in $S$ and $\{0,1\}^d\setminus S$ as \emph{feasible} and \emph{infeasible} points, respectively. Denote by $G_d$ the graph on vertex set $\{0,1\}^d$ where two vertices are adjacent if they are at Hamming distance $1$.
	
	Given weights $w\in \bZ_{\geq 0}^{[d]\cup [\bar{d}]}$ defined on the ground set of $\cuboid(S)$, a \emph{lopsided twisting} of $(S,w)$ is a twisting of the coordinates, along with their corresponding weights, such that $w_i\leq w_{\bar{i}}$ for all $i\in [d]$.
	
\begin{lemma} \label{lem:zero-infeasible}
	Let $S \subseteq \{0,1\}^d$ be a polar set-system and let $w\in \bZ^{[d]\cup [\bar{d}]}_{\ge 0}$ be weights satisfying the following conditions:
	\begin{enumerate}[label=(h\arabic*)]
			\item \label{it:h1} $\tau\coloneqq\tau(\cC, w) > \nu(\cC, w)$,
			\item \label{it:h2} $w$ is entry-wise minimal subject to satisfying \ref{it:h1}, 
			\item \label{it:h3} 
			for every single restriction $S'\subseteq \{0,1\}^{d-1}$ of $S$, and all weights $w'\in \mathbb{Z}^{2d-2}_{\ge 0}$ such that $\tau(\cC,w') = \tau$, we have $\tau(\cC,w')=\nu(\cC,w')$.
		\end{enumerate} Then $(S,w)$ has a lopsided twisting such that $\0 \in \bar{S}$. 
\end{lemma}
\begin{proof}
	Consider a lopsided twisting of $(S,w)$ and let $\cC \coloneqq \cuboid(S)$. 
	Assume that $\0 \in S$, that is, $[\bar{d}]\coloneqq\{\bar{1},\ldots,\bar{d}\}\in \cC$.
	Let $B$ be a mate of $[\bar{d}]$.
	As $\tau \le w(B) \le \tau-2+|B\cap [\bar{d}]|$, 
	we have $|B\cap [\bar{d}]|\ge 2$, thus $B$ contains at least two elements of weight at least $\tau/2$.
	Then, \cref{prop:acl}\ref{it:mate} implies that it contains exactly two, say $f$ and $g$, $B\subseteq [\bar{d}]$ and $w_e=1$ for all $e \in B\setminus \{f, g\}$.
	It follows that $B = \{\bar{i}, \bar{j}\}$ for some $i, j \in [d]$ with $w_i = w_j = \tau/2$.
	In particular, $e_i+e_j \in \bar{S}$.
	Lopsidedly twisting both the $i$th and $j$th coordinates maps $e_i+e_j$ to $\0$, thus $\0\in \bar{S}$ after twisting.
\end{proof}
	
\begin{thm} \label{thm:degree-dimension}
	Let $S\subseteq \{0,1\}^d$ be a polar set-system, let $\cC\coloneqq\cuboid(S)$, and let $w\in \mathbb{Z}_{\ge 0}^{[d]\cup [\bar{d}]}$ be weights satisfying \ref{it:h1}, \ref{it:h2}, and \ref{it:h3}. 
	Then $S$ has degree $d$, that is, there is an infeasible point with no feasible neighbor.
\end{thm}
\begin{proof}
	Assume to the contrary that $S$ has degree at most $d-1$.
	We assume that the current twisting of $(S,w)$ is lopsided, that is, $w_i\leq w_{\bar{i}}$ for all $i\in [d]$. 
	We will show by induction that for any integer $l\in \{0,1,\ldots,d\}$, and for some lopsided twisting of $(S,w)$, there exist covers $B_1, \dots, B_l$, and pairwise distinct indices $j_1, \dots, j_l \in [d]$ such that $\0 \in \bar{S}$ and for each $i \in [l]$,
	\begin{align} \label{eq:new_ind}
		w(B_i)=\tau, ~ B_i\cap [\bar{d}] = \{\bar{j_i}\}, ~B_i\cap \{j_1,\dots, j_l\} = \emptyset. 
	\end{align}
	
	The base case $l=0$, that $\0 \in \bar{S}$ can be achieved after a possible lopsided twisting, follows from \Cref{lem:zero-infeasible}. 
		
	Assume next for contradiction that the statement holds for some $d-1\geq l \ge 0$ and some lopsided twisting of $(S,w)$, but fails for $l+1$. Thus we have $0 \in \bar{S}$ and indices $j_1,\dots, j_l$, and covers $B_1,\dots, B_l$ satisfying \eqref{eq:new_ind}.
	We may assume that $j_1=1, \dots, j_l = l$. (It is possible that $l=0$.) In what follows, we prove a series of claims on covers and mates that will be used multiple times.
		
	\begin{claim} \label{cl:disjoint_mate}
		Each $C \in \cC$ has a mate $B$ with $B\cap [l] = \emptyset$.
	\end{claim}
	\begin{claimproof}
		Let $B$ be a mate of $C$ such that $|B\cap [l]|$ is minimal.
		Assume to the contrary that $|B\cap [l]| \ge 1$ and take an index $i \in B\cap [l]$.
		Since $B \cap \{i, \bar{i}\} = \{i\}$, $B_i \cap \{i, \bar{i}\} = \{\bar{i}\}$, and both $B$ and $B_i$ are covers, it follows from the resolution principle that $B'\coloneqq (B\cup B_i) \setminus \{i, \bar{i}\}$ is a cover as well.
		We have 
		\[w(B') = w(B)+w(B_i)-w(B\cap B_i) -w_i-w_{\bar{i}} = w(B)-w(B\cap B_i)\] and 
		\begin{align*}
			|B'\cap C| & = |B\cap C| + |B_i \cap C| - |B\cap B_i \cap C| - |\{i, \bar{i}\} \cap C| \\
			& \ge |B \cap C| + 1 - |B\cap B_i| - 1 \\
			& \ge |B\cap C| - w(B\cap B_i),
		\end{align*}
		hence
		\[w(B')-|B'\cap C| \le w(B)-|B\cap C| \le \tau -2.\]
		Take a minimal cover $B''\subseteq B'$.
		Then, \[w(B'')-|B''\cap C| \le w(B')-|B'\cap C| \le w(B)-|B\cap C| \le \tau -2.\]
		Therefore, $B''$ is a mate of $C$ with $|B''\cap [l]| \le |B'\cap [l]| = |B\cap [l]|-1$, a contradiction.
	\end{claimproof}
		
	\begin{claim} \label{cl:no_disjoint_covers}
		Suppose $B$ and $B'$ are covers such that $|B\cap \{i, \bar{i}\} | \le 1$ and $|B'\cap \{i, \bar{i}\}| \le 1$ for each $i \in [d]$.
		Then there exists $i \in [d]$ with $|B\cap \{i, \bar{i}\}| = |B'\cap \{i, \bar{i}\}| = 1$.
	\end{claim}
	\begin{claimproof}
		Assume to the contrary that $|B\cap \{i,\bar{i}\}| + |B'\cap \{i, \bar{i}\}| \le 1$ for all $i \in [d]$.
		Then there exists $p \in \{0,1\}^d$ such that the set corresponding to $p$ is disjoint from both of $B$ and $B'$.
		Since $B$ and $B'$ are covers, we have $p \in \bar{S}$. 
		For each $i \in [d]$, $|B\cap \{i,\bar{i}\}| + |B'\cap \{i, \bar{i}\}| \le 1$ implies that the set corresponding to $p\triangle e_i$ is disjoint from at least one of $B$ and $B'$, thus $p\triangle e_i \in \bar{S}$.
		Therefore, $p$ has degree $d$ in $G_d[\bar{S}]$, a contradiction as $S$ has degree at most $d-1$.
	\end{claimproof}
	
	\begin{claim} \label{cl:size_two_cover} 
		Suppose $f, g \in [d] \cup [\bar{d}]$ are distinct indices such that $f\neq \bar{g}$, $w_f = w_g = \tau/2$, and $B\coloneqq \{f,g\}$ is a cover. Then $\{f, \bar{f}, g, \bar{g}\} \cap (B_1 \cup \dots \cup B_l) \neq \emptyset$. In particular, $\ell\geq 1$.
	\end{claim}
	\begin{claimproof}
		Assume to the contrary that $\{f, \bar{f}, g, \bar{g}\} \cap (B_1 \cup \dots \cup B_l) = \emptyset$.
		Let $i, j \in [d]$ be the indices such that $f \in \{i, \bar{i}\}$ and $g \in \{j, \bar{j}\}$. By assumption, $i\neq j$. 
		Since $\{f, \bar{f}, g, \bar{g}\} \cap (B_1 \cup \dots \cup B_l) = \emptyset$, we have $i, j \ge l+1$.
		If $|B \cap [d]| = |B\cap [\bar{d}]| = 1$, say $f=i$ and $g=\bar{j}$, then as $i\neq j$, \eqref{eq:new_ind} is satisfied for the indices $1, \dots, l, j$ and covers $B_1, \dots, B_l, B$, a contradiction.
		Otherwise, we have $B=\{i,j\}$ or $B = \{\bar{i}, \bar{j}\}$.
		If $e_i \in \bar{S}$ or $e_j \in \bar{S}$, then lopsidedly twisting the $i$th or $j$th coordinate, respectively, maintains $\0 \in \bar{S}$ and yields $|B\cap [d]| =|B\cap [\bar{d}]| = 1$, thus we get the same contradiction as in the previous case.
			
		It remains to deal with the case $e_i, e_j \in S$.
		
		As $e_i \in S$, the corresponding set $C = [\bar{d}]\triangle \{i,\bar{i}\} \in \cC$ has a mate $B'$.
		As $w_e \ge \tau/2$ for each $e \in C$, we have $B'\subseteq C$, $|B'|=2$, and $w_e = \tau/2$ for $e \in B'$.			\cref{cl:no_disjoint_covers} applied to $B$ and $B'$ shows that $B'\cap \{i, \bar{j}\}\ne \emptyset$.
		We have $B'\ne \{i, \bar{j}\}$, as otherwise either both $\{i, \bar{j}\}$ and $\{i,j\}$ or both $\{i, \bar{j}\}$ and $\{\bar{i}, \bar{j}\}$ are covers, thus by the resolution principle, either $\{i\}$ or $\{\bar{j}\}$ would be a cover of size one, a contradiction.
		We get that $B'= \{i,\bar{s}\}$ or $B'= \{\bar{j}, \bar{s}\}$ for some $s \in [d]\setminus \{i,j\}$ with $w_s = \tau/2$.
		If $B=\{i,j\}$ and $B'=\{\bar{j}, \bar{s}\}$, then $(B\cup B') \setminus \{j,\bar{j}\}= \{i, \bar{s}\}$ is a cover, while if $B=\{\bar{i},\bar{j}\}$ and $B'=\{i, \bar{s}\}$, then $(B\cup B') \setminus \{i, \bar{i}\} = \{\bar{j}, \bar{s}\}$ is a cover, by the resolution principle. 
		We conclude that $\{i, \bar{s}\}$ is a cover in all cases with $B=\{i,j\}$, and $\{\bar{j}, \bar{s}\}$ is a cover in all cases with $B=\{\bar{i}, \bar{j}\}$.
		
		By using a similar argument with $e_j\in S$ instead of $e_i$, we get that for some $t \in [d]\setminus \{i,j\}$, either all of $\{i,j\}, \{i, \bar{s}\}, \{j, \bar{t}\}$ or all of $\{\bar{i}, \bar{j}\}, \{\bar{j}, \bar{s}\}, \{\bar{i}, \bar{t}\}$ are covers. Now take a point $p \in \{0,1\}^d$ such that the set corresponding to $p$ is disjoint from all three covers.
		For any such $p$ and any $k \in [d]$, $p \triangle e_k$ is disjoint from at least one of the three covers since no element is contained in their intersection.
		This shows that $p$ is an infeasible point with all infeasible neighbors, contradicting $S$ having degree at most $d-1$. 	
	\end{claimproof}
	
	\begin{claim} \label{cl:special_mate}
		Assume that $C \in \cC$ satisfies $[l] \subseteq C$ and $C\setminus [l]$ contains at  most one element with weight strictly less than $\tau/2$.
		Then $C$ has a mate $B$ with $B\cap [l] = \emptyset$ such that one of the following holds:
		\begin{enumerate}[label=(\roman*)]
			\item \label{it:bad_size_two} $B = \{f, g\} \subseteq C$ for some $f, g \in \{l+1, \bar{l+1}, \dots,d, \bar{d}\}$ with $w_f = w_g = \tau/2$, 
			\item \label{it:bad_tau_plus_one} $B=\{j,f,g\}\subseteq C$ for some $j \in \{l+1, \dots,d\}$ and distinct $f, g \in \{l+1, \bar{l+1}, \dots,d, \bar{d}\}$ such that $w_j=1$ and $w_f=w_g=\tau/2$, 
			\item \label{it:bad_tau} $w(B) = \tau$, $B\cap C = \{j, s\}$ for some $j, s \in [d] \setminus[l]$ such that $w_j < \tau/2$ and $w_s = \tau/2$, and there exists $i \in [l]$ with $w_i=\tau/2$ and $\{\bar{i}, s\}$ being a cover. 
		\end{enumerate}
	\end{claim}
	\begin{claimproof}
		By \cref{cl:disjoint_mate}, $C$ has a mate $B$ with $B\cap [l] = \emptyset$.
		Assume first that $B$ contains at least two elements of weight at least $\tau/2$.
		Then \cref{prop:acl}, part~\ref{it:mate} implies that $B\subseteq C$, $B$ contains two elements of weight $\tau/2$, say $f$ and $g$, and $w_e = 1$ for each $e \in B\setminus \{f,g\}$.
		Since $B\cap [l] = \emptyset$ and $C\setminus [l]$ contains at most one element with weight less than $\tau/2$, we have $B=\{f, g\}$ or $B=\{j, f, g\}$ for some $j \in [d] \setminus [l]$ with $w_j = 1$. 
		This shows that \ref{it:bad_size_two} or \ref{it:bad_tau_plus_one} holds.
		
		It remains to deal with the case when $B$ contains at most one element of weight at least $\tau/2$. 
		As $B\cap [l] = \emptyset$ and $C\setminus [l]$ contains at most one element with weight less than $\tau/2$, we have $|B\cap C|\le 2$.
		Then, $\tau \le w(B) \le |B\cap C|+ \tau-2 \le \tau$, hence equality holds throughout.
		This implies that $w(B)=\tau$ and $B\cap C=\{j, e\}$ for the unique $j \in B\cap C$ with $w_j < \tau/2$ and for some $e \in C$ with $w_e \ge \tau/2$.
		We may assume that $e \in \{l+1, \bar{l+1}\}$, and note that this implies $j \ge l+2$.
		
		We claim that $e=l+1$. For if not, then $e = \bar{l+1}$, and we will show that $B_1,\dots, B_{l+1}$ satisfy \eqref{eq:new_ind} by setting $B_{l+1} \coloneqq B$, $j_1 \coloneqq 1, \dots, j_{l+1} \coloneqq l+1$, thus yielding a contradiction. 
		We already have that $B_1, \dots, B_{l+1}$ are covers with $w(B_1) = \dots = w(B_{l+1}) = \tau$, $B_i \cap [\bar{d}] = \{\bar{i}\}$ for $i \in [l]$, $B_{l+1} \cap [l+1] = \emptyset$ and $B_i \cap [l] = \emptyset$ for $i \in [l]$.
		It only remains to check that $l+1 \not \in B_i$ for $i \in[l]$.
		Assume to the contrary that $l+1 \in B_i$ for some $i \in [l]$.
		As $B_i \cap \{l+1, \bar{l+1}\} = \{l+1\}$ and $B \cap \{l+1, \bar{l+1}\}=\{\bar{l+1}\}$, $(B_i \cup B_{l+1}) \setminus \{l+1, \bar{l+1}\}$ is a cover by the resolution principle, hence \[\tau \le w((B_i \cup B)\setminus \{l+1, \bar{l+1}\}) = w(B_i) + w(B) - w(B\cap B_i) - w_{l+1} - w_{\bar{l+1}} = \tau - w(B\cap B_i),\]
		thus $B\cap B_i = \emptyset$.
		In particular, we have $j \not \in B_i$.
		Since $B_i \cap C \ne \emptyset$, $B_i\cap ([\bar{d}] \cup [l+1]) = \{\bar{i}, l+1\}$ and $C\cap \{\bar{i}, l+1\} = \emptyset$, we get that $s \in B_i \cap C$ for some $s \ge l+2$ with $s \ne j$.
		As each element of $C\setminus ([l] \cup \{j\})$ has weight at least $\tau/2$, we have
		\[w(B_i) \ge w_{\bar{i}} + w_{l+1} + w_{s} \ge \frac{\tau}{2} + w_{l+1} + \frac{\tau}{2} > \tau,\]
		a contradiction.
		This shows that $B_1, \dots, B_{l+1}$ indeed satisfy \eqref{eq:new_ind}, a contradiction.
		Thus, $e=l+1$.
		
		In summary, we have $B\cap C = \{j, l+1\}$ and $w_{l+1} = \tau/2$. In what follows, we show that \ref{it:bad_tau} holds for $s=l+1$ and some $i\in [l]$.
		
		If $l+1 \in B_i$ for some $i \in [l]$, then $\tau = w(B_i) \ge w_{\bar{i}} + w_{l+1} + w(B_i\setminus \{\bar{i}, l+1\}) \ge \tau/2 + \tau/2 + 0$, hence $w_{\bar{i}} = \tau/2$ and $B_i = \{\bar{i}, l+1\}$.
		This shows that \ref{it:bad_tau} holds for $s=l+1$.
		
		Otherwise, we have that $l+1 \not \in B_1 \cup \dots \cup B_l$.
		As $l+1$ is the only element of $B$ with weight at least $\tau/2$, we have $B\cap [\bar{d}] = \emptyset$.
		Therefore, $B$ is a minimum weight cover with $B\cap ([\bar{d}] \cup [l+1]) = \{l+1\}$, $w_{l+1} = \tau/2$ and $l+1 \not \in B_1 \cup \dots \cup B_l$. If $e_{l+1}\in \overline{S}$, then \eqref{eq:new_ind} holds for $B_1, \dots, B_l, B$ after lopsidedly twisting the $(l+1)$st coordinate, as $\0$ is infeasible after the twisting, a contradiction. Thus, $e_{l+1} \in S$. 
		
		Consider the set $C' \coloneqq [\bar{d}]\triangle \{l+1,\bar{l+1}\} \in \cC$ that corresponds to $e_{l+1}$, and let $B'$ be a mate of $C'$.
		As $w_e \ge \tau/2$ for each $e \in C'$, \cref{prop:acl}, part~\ref{it:mate} shows that $B'\subseteq C'$, $|B'|=2$ and $w_e = \tau/2$ for $e \in B'$.
		\cref{cl:no_disjoint_covers} shows that $|B\cap \{t, \bar{t}\}| = |B'\cap \{t, \bar{t}\}| = 1$ for some $t \in [d]$.
		As we have $w_t = \tau/2$ for each $t \in [d]$ with $|\{t, \bar{t}\} \cap B'| = 1$, and $w_t < \tau/2$ for each $t \in [d]\setminus \{l+1\}$ with $|\{t, \bar{t}\} \cap B| = 1$, we have that $t=l+1$ and so $l+1 \in B\cap B'$.
		Then $B' = \{\bar{i}, l+1\}$ for some $i \in [d]$ with $w_i = \tau/2$. Note that $i\neq l+1$ as $B'\subseteq C'$. 
		
		As $B'$ is a cover, it follows from \cref{cl:size_two_cover} that $\{i,\bar{i},l+1,\bar{l+1}\}\cap (B_1\cup \cdots\cup B_l)\neq \emptyset$. As $B_1\cup \cdots\cup B_l$ does not contain the indices $l+1,\bar{l+1}$, it must therefore contain one of $i,\bar{i}$. If it is the former, then $i\in B_r$ for some $r\in [l]$, so $B_r=\{i,\bar{r}\}$ and $w_{\bar{r}}=\tau/2$, as $w(B_r)=\tau$, $w_i=\tau/2$ and $w_{\bar{r}}\geq \tau/2$. However, as $l+1$ is the only element of $B$ with weight at least $\tau/2$, and $B_r\cap \{l+1,\bar{l+1}\}=\emptyset$, it follows that $|B\cap \{t,\bar{t}\}|+|B_r\cap \{t,\bar{t}\}|\leq 1$ for all $t\in [d]$, a contradiction to \cref{cl:no_disjoint_covers}. Thus, $B_1\cup \cdots\cup B_l$ must contain $\bar{i}$, i.e., $i \in [l]$. As $B'=\{\bar{i},l+1\}$ is a cover, it follows that \ref{it:bad_tau} holds for $s=l+1$, as claimed.
	\end{claimproof}
	
	Using these claims, we will construct sets in the cuboid and their mates. 
	First observe that
	\begin{align} \label{eq:inf}
		\{x \in \{0,1\}^d {}\mid{} x_i = 0\text{ for $i \ge l+1$}\} \subseteq \bar{S}.
	\end{align}
	Indeed, if $x \in \{0,1\}^d$ satisfies that $x_i=0$ for all $i \ge l+1$, then either $x=0$, or $x_j =1$ for some $j \le l$ in which case the set corresponding to $x$ is disjoint from the cover $B_j$; in either case, we have $x\in \bar{S}$. 
		
	Consider now the point $e_1+\dots + e_l$ which is in $\bar{S}$ by \eqref{eq:inf}.
	As $S$ has degree at most $d-1$, the infeasible point $e_1+\dots + e_l$ has a feasible neighbor, which by \eqref{eq:inf} must be of the form $e_1+\dots + e_l + e_{j_0} \in S$ for some $j_0 \ge l+1$.
	Let $C_0\in \cC$ denote the set corresponding to $e_1+\dots + e_l+e_{j_0}$. 
	For each $i\in [l]$, $B_i$ is a cover such that $B_i\cap [\bar{d}] = \{\bar{i}\}$, so $j_0\in B_i$ as $B_i\cap C_0\neq \emptyset$. Thus, $j_0 \in B_1 \cap \dots \cap B_{l}$. 
	Note that each element of $C_0\setminus ([l] \cup \{j_0\})$ has weight at least $\tau/2$, thus we can apply \cref{cl:special_mate} to get a mate $B_0$ of $C_0$ satisfying \ref{it:bad_size_two}, \ref{it:bad_tau_plus_one}, or \ref{it:bad_tau}.
	
	Case \ref{it:bad_tau} is impossible since we would have $|C\cap ([d] \setminus [l])|\geq 2$, which is not the case.
	
	Assume now that \ref{it:bad_size_two} holds. 
	Then either $w_{j_0} = \tau/2$ and $B_0 = \{j_0, \bar{s}\}$ for some $s \in [d] \setminus ([l] \cup \{j\})$ with $w_s = \tau/2$, or $B_0 = \{\bar{s}, \bar{t}\}$ for some $s, t \in [d] \setminus ([l] \cup \{j\})$ with $w_s = w_t = \tau/2$.
	If the former holds, then for each $i \in [l]$, 
	$w(B_i) = \tau$ and $ \{\bar{i}, j_0\} \subseteq B_i$ together imply that $w_i = \tau/2$ and $B_i = \{\bar{i}, j_0\}$.
	Then $B_1, \dots, B_l, B_0$ satisfy \eqref{eq:new_ind} with $j_1 = 1, \dots, j_l = l, j_{l+1} = s$, a contradiction as we have assumed that \eqref{eq:new_ind} fails for $l+1$. 
	Otherwise, we have $B_0 = \{\bar{s}, \bar{t}\}$ for some $s, t \in [d] \setminus ([l] \cup [j])$ with $w_s = w_t = \tau/2$. 
	As $(B_1 \cup \dots \cup B_l) \cap [\bar{d}] = [\bar{l}]$, \cref{cl:size_two_cover} shows that $l\geq 1$ and $\{s, t\} \cap B_i \ne \emptyset$ holds for some $i \in [l]$, say $s \in B_i$. 
	Then $\tau = w(B_i) \ge w_{\bar{i}} + w_{j_0} + w_s \ge \tau/2 + 1 + \tau/2 = \tau+1$, a contradiction.
	
	We conclude that \ref{it:bad_tau_plus_one} holds for $B_0$.
	To summarize, there exist $f_0, g_0 \in \{\bar{l+1}, \dots,\bar{d}\}\setminus \{\bar{j_0}\}$ such that
	\begin{align} \label{eq:j0}
		j_0 \in B_1 \cap \dots \cap B_l, \quad  w_{j_0} = 1, \quad B_0 = \{j_0, f_0, g_0\}, \text{ and } w_{f_0} = w_{g_0} = \tau/2. 
	\end{align} In particular, $\tau$ must be even.
	
	Let $k \coloneqq |\{s \in [d]\setminus [l] \mid w_s = \tau/2\}|$.	
	To simplify notation, we may assume that 
	\begin{align} \label{eq:k}
		w_{l+1}, \dots, w_{d-k} < \frac{\tau}{2} \text{ and } w_{d-k+1} = \dots = w_d = \frac{\tau}{2}.
	\end{align}
	As $w_{j_0} = 1$, and $\tau>2$, we have $l+1 \le j_0 \le d-k$.
	If $s \in B_i$ for some $i \in [l]$ and $d-k+1 \le s \le d$, then $\tau = w(B_i) \ge w_{\bar{i}} + w_{j_0} + w_s \ge \tau/2 + 1 + \tau/2 = \tau+1$,  contradiction.
	We conclude that
	\begin{align} \label{eq:Bi_where}
		B_i \subseteq \{\bar{i}\} \cup \{l+1, \dots,d-k\} \qquad \forall i\in [l].
	\end{align}

	Next we use \eqref{eq:Bi_where} to strengthen \eqref{eq:inf} as follows.	
	
	\begin{claim} \label{cl:middle_cover}
		$\{l+1, \dots,d-k\}$ is a cover.
	\end{claim}
	\begin{claimproof}
		Assume to the contrary that there exists $p \in S$ with $p_{l+1} = \dots = p_{d-k} = 0$.
		Let $C \in \cC$ be the set corresponding to $p$.
		By \cref{cl:disjoint_mate}, $C$ has a mate $B$ with $B\cap [l] = \emptyset$.
		Since $w_e \ge \tau/2$ for each $e \in C\setminus [l]$, $B\cap [l]=\emptyset$, and $|B\cap C|\geq 2$, \cref{prop:acl}, part~\ref{it:mate} shows that $B=\{f, g\} \subseteq C$ for some $f, g$ with $w_f = w_g = \tau/2$. By \eqref{eq:k}, we have $f, g \in [\bar{l}] \cup \{d-k+1, \bar{d-k+1}, \dots,d, \bar{d}\}$. 
		As $e_1+\dots + e_l + e_{j_0} \in S$ where $l+1 \le j_0 \le d-k$ and $B$ is a cover, we get that $B\cap \{\bar{d-k+1}, \dots, \bar{d}\} \ne \emptyset$, say $g= \bar{d}$. It follows from \cref{cl:size_two_cover} that $(B_1 \cup \dots \cup B_l) \cap \{f, \bar{f}, g, \bar{g}\} \ne \emptyset$, so $l\geq 1$. By \cref{cl:no_disjoint_covers} applied to $B,B_1$, and by \eqref{eq:Bi_where}, we have that $f=\bar{1}$. Furthermore, \cref{cl:no_disjoint_covers} applied to $B,B_0$ implies that $g\in \{f_0,g_0\}$, say $g=g_0$. Subsequently, $B\cap B_0\cap B_1=\emptyset$. Let $q\coloneqq \sum_{i=1}^l e_i + e_d$. The set corresponding to $q$ is disjoint from all the covers $B, B_0, B_1$, and given that $B\cap B_0\cap B_1=\emptyset$, the set corresponding to every neighbor of $q$ is disjoint from at least one of $B, B_0, B_1$. This implies that $q$ and all its neighbors are infeasible, a contradiction as $S$ has degree at most $d-1$. 
	\end{claimproof}
		
	Previously we used \eqref{eq:inf} and \cref{cl:special_mate} to get a cover of the form \ref{it:bad_tau_plus_one}.
	Now we generalize this to get more covers of this form.
	
	\begin{claim} \label{cl:transversals}
		For each $T\subseteq \{d-k+1, \bar{d-k+1}, \dots,d, \bar{d}\}$ with $|T\cap \{d-k+1, \bar{d-k+1}\}| = \dots = |T\cap \{d, \bar{d}\}| = 1$, there exist $j \in \{l+1, \dots,d-k\}$ and distinct $f, g \in T$ such that $f\ne \bar{g}$, $w_j = 1$ and $\{j, f, g\}$ is a cover.
	\end{claim}
	\begin{claimproof}
		Let $p \coloneqq \sum_{i=1}^l e_i + \sum_{t \in \{d-k+1, \dots,d\}\cap T} e_t$.
		By \cref{cl:middle_cover}, we have $p \in \bar{S}$ and $p \triangle e_j \in \bar{S}$ for each $j \in [d] \setminus \{l+1, \dots,d-k\}$.
		As $p$ has degree at most $d-1$ in $G_d[\bar{S}]$, we get that $p \triangle e_j \in S$ for some $j \in \{l+1, \dots,d-k\}$.
		Let $C$ denote the set corresponding to $p$ and observe that $T\subseteq C$.
		As each element of $C\setminus ([l] \cup \{j\})$ has weight at least $\tau/2$, \cref{cl:special_mate} gives a mate $B$ of $C$ such that $B\cap [l]=\emptyset$ and satisfying \ref{it:bad_size_two}, \ref{it:bad_tau_plus_one}, or \ref{it:bad_tau}.
		In case of \ref{it:bad_size_two}, $B = \{f, g\}$ for some $f, g \in T$.
		Then $(B_1 \cup \dots \cup B_l) \cap \{f, \bar{f}, g, \bar{g}\} \ne \emptyset$ by \cref{cl:size_two_cover}, contradicting \eqref{eq:Bi_where}.
		If \ref{it:bad_tau} holds, then $\{\bar{i}, s\}$ is a cover for some $i \in [l]$ and $s \in \{d-k+1, \dots,d\}$.
		This is a contradiction, since $e_1+\dots + e_l + e_{j_0} \in S$ where $l+1 \le j_0 \le d-k$.
		We showed that $B$ satisfies \ref{it:bad_tau_plus_one}, that is, $B=\{j, f, g\}\subseteq C$ where $f, g \in \{d-k+1, \bar{d-k+1}, \dots,d, \bar{d}\}$ and $j\in \{l+1,\ldots,d-k\}$ satisfies $w_j=1$.
		Thus, $f, g \in T$ and the claim follows. 
	\end{claimproof}
	
	Let $V\coloneqq \{d-k+1, \bar{d-k+1}, \dots,d, \bar{d}\}$ and define 
	\[\mathcal{D} \coloneqq \left\{(j, \{f, g\}) \mid j \in \{l+1, \dots,d-k\}, f, g \in V, f \not \in \{g, \bar{g}\}, w_j = 1, \text{$\{j,f,g\}$ is a cover}\right\}.\]
	\cref{cl:transversals} equivalently states there exists no $T\subseteq V$ with $|T\cap \{d-k+1, \bar{d-k+1}\}| = \dots = |T\cap \{d, \bar{d}\}| = 1$ such that $\{\bar{f}, \bar{g}\}\cap T \ne \emptyset$ for all $(j, \{f, g\}) \in \cD$. 
	This states the unsatisfiability of a 2-SAT formula as follows. 
	Consider a set of variables $X\coloneqq \{x_{d-k+1}, \dots, x_d\}$, and let $L$ denote the set of literals over $X$.
	Define the mapping $\varphi\colon V \to L$ by $\varphi(i) = x_i$ and $\varphi(\bar{i}) = \bar{x}_i$ for $d-k+1 \le i \le d$.
	Then our previous formulation of \cref{cl:transversals} is equivalent to the unsatisfiability of the 2-SAT formula with variables $X$ and set of clauses
	\[\left\{\left(\bar{\varphi(f)} \vee \bar{\varphi(g)}\right) \mid (j, \{f,g\}) \in \cD\right\}.\]
	By \cref{prop:2sat}, this can be characterized using the implication graph of the instance.
	Consider a directed graph $D=(V, A)$ with the coloring $c\colon A \to \{d-k+1, \dots,d\}$ of its arcs such that an arc $f\bar{g}$ of and an arc $g\bar{f}$ are added to $A$ with $c(f\bar{g}) = c(g\bar{f}) = j$ for each $(j, \{f, g\}) \in \cD$, where we only add one of the arcs with the same start and end vertex (arbitrarily deciding its color among the feasible choices). 
	Then $D$ is isomorphic to the implication graph of the above 2-SAT formula.
	As the instance is unsatisfiable, \cref{prop:2sat} states that there exist both a directed $s\bar{s}$-path and a directed $\bar{s}s$-path in $D$ for some $s \in \{d-k+1, \dots,d\}$.
	We observe the following property of directed paths in $D$.
	
	\begin{claim} \label{cl:path}
		Let $P$ be a directed $u_1u_2$-path in $D$ for some $u_1, u_2 \in V$.
		Then $\{c(vw) \mid vw \in P\} \cup \{u_1, \bar{u_2}\}$ is a cover (of $\cC$).
	\end{claim}
	\begin{claimproof}
		We prove the claim by induction on the length $m$ of $P$. 
		If $m=0$, then $\{u_1, \bar{u_2}\} = \{u_1,\bar{u_1}\}$ is indeed a cover. 
		Assume now that the statement holds for all paths of length $m \ge 0$; we prove that it holds if $P$ has length $m+1$. 
		Let $zu_2$ denote the last arc of $P$, and let $P'$ be the $u_1z$-path obtained from $P$ by removing its last arc. 
		By the induction hypothesis, 
		\[B'\coloneqq \{c(vw) \mid vw \in P'\} \cup \{u_1, \bar{z}\}\] 
		is a cover. 
		As $zu_2$ is an arc of $D$, we have that $B''\coloneqq \{c(zu_2), z, \bar{u_2}\}$ is a cover. 
		Therefore, after applying the resolution principle to $B',B''$ by resolving $z$, we obtain that $\{c(vw) \mid vw \in P\} \cup \{u_1, \bar{u_2}\}$ is a cover. 
	\end{claimproof}

	We are ready to finish the proof. 
	Take a shortest directed path in $D$ that starts from some $u\in V$ and ends at $\bar{u}$. 
	By definition, $P$ does not contain both $v$ and $\bar{v}$ for any $v \in V\setminus \{u\}$.
	
	Consider the case first when $P$ has length two, that is, $P=\{uv, v\bar{u}\}$ for some $v\in V$. \cref{cl:path} shows that $B \coloneqq \{c(uv), c(v\bar{u}), u\}$ is a cover.
	Then $\tau \le w(B) \le 2+\tau/2$, hence $w(B)=\tau=4$, as $\tau>2$ is even.
	If $u \in [\bar{d}]$, then $B_1, \dots, B_l, B$ satisfy \eqref{eq:new_ind} with $j_1 = 1, \dots, j_l = l, j_{l+1} = u$, a contradiction.
	Otherwise, if $u \in [d]$, then $e_u \in \bar{S}$ by \cref{cl:middle_cover}, thus we get the same contradiction after twisting the $u$th coordinate.
	
	In the remaining case, $P$ has length at least three. 
	Let $v_1, v_2 \in V$ be the vertices such that $uv_1$ is the first and $v_2\bar{u}$ is the last arc of $P$. Note that $v_1 \notin \{v_2,\bar{v_2}\}$ as by the minimal choice of $P$ and the lower bound on its length. 
	By the definition of $D$, $B\coloneqq \{c(uv_1), u, \bar{v_1}\}$ and $B'\coloneqq \{c(v_2\bar{u}), v_2, u\}$ are covers.
	Let $p \in \{0,1\}^d$ be the point such that $p_i = 0$ for $i \in [d] \setminus \{v_1, \bar{v_1}, v_2, \bar{v_2}\}$ and the set corresponding to $p$ is disjoint from $B \cup B'$.
	As $B$ and $B'$ are covers, we have $p \in \bar{S}$.
	Since $S$ has degree at most $d-1$, $p$ has a feasible neighbor, say $q\in S$ with corresponding set $C\in \cC$.
	By \cref{cl:middle_cover}, $C$ intersects each of $B$, $B'$, and $\{l+1, \dots,d-k\}$, while the set corresponding to $p$ intersects none of these three sets, hence $B \cap B' \cap \{l+1, \dots,d-k\} \ne \emptyset$.
	This shows that $c(uv_1) = c(v_2\bar{u})=:c_0\in C$ and $q=e_{c_0}$. 	
	
	Consider now any arc $w_1w_2 \in P\setminus \{uv_1, v_2\bar{u}\}$.
	By the definition of $D$, $\{c(w_1w_2), w_1, \bar{w_2}\}$ is a cover.
	Assume first that $c(w_1w_2) \ne c_0$.
	Since $\{c_0, u, \bar{v_1}\}$ and $\{c(w_1w_2), w_1, \bar{w_2}\}$ are covers, \cref{cl:no_disjoint_covers} shows that $\{w_1, w_2\} \cap \{u, \bar{u}, v_1, \bar{v_1}\} \ne \emptyset$.
	The minimality of $P$ implies that $\{w_1, w_2\} \cap \{u, \bar{u},\bar{v_1}\}= \emptyset$, so $v_1\in \{w_1,w_2\}$, implying that $w_1=v_1$. 
	Similarly, since $\{c_0, v_2, \bar{u}\}$ and $\{c(w_1w_2), w_1, \bar{w_2}\}$ are covers, \cref{cl:no_disjoint_covers} shows that $\{w_1, w_2\} \cap \{u, \bar{u}, v_2, \bar{v_2}\} \ne \emptyset$.
	The minimality of $P$ implies that $\{w_1, w_2\} \cap \{u, \bar{u},\bar{v_2}\}= \emptyset$, so $v_2\in \{w_1,w_2\}$, implying that $w_2=v_2$. 
	We showed that $w_1w_2 = v_1v_2$, that is, $P=\{uv_1, v_1v_2, v_2\bar{u}\}$.
	Then \cref{cl:path} shows that $B''\coloneqq \{c_0, c(v_1v_2), u\}$ is a cover.
	Since $\tau \le w(B'') \le 2 + \tau/2$, we get that $w(B'') = \tau=4$.
	If $u \in [\bar{d}]$, then $B_1, \dots, B_l, B''$ satisfy \eqref{eq:new_ind} with $j_1 = 1, \dots, j_l = l, j_{l+1} = u$, a contradiction.
	Otherwise, if $u \in [d]$, then $e_u \in \bar{S}$ by \cref{cl:middle_cover}, thus we get the same contradiction after twisting the $u$th coordinate.
	We showed that having $c(w_1w_2) \ne c_0$ for some $w_1w_2\in P \setminus \{uv_1, v_2\bar{u}\}$ leads to a contradiction, hence $c(w_1w_2) = c_0$ for each $w_1w_2 \in P$.
	Then \cref{cl:path} shows that $\{c_0, u\}$ is a cover, yielding a contradiction as its weight is $1+\tau/2 < \tau$.
	
	We showed by induction that \eqref{eq:new_ind} holds for each $l \in [d]$.
	For $l=d$, \eqref{eq:new_ind} shows that $B_i = \{\bar{j_i}\}$ for $i \in [d]$.
	This is clearly a contradiction.
	\end{proof}

\begin{proof}[Proof of \Cref{non-MFMC-cuboid}]
Let $S$ be a strictly polar set-system of degree $\delta$ with a non-MFMC cuboid. Let $S'\subseteq \{0,1\}^d$ be a restriction of $S$ with a non-MFMC cuboid that is of the minimum dimension $d$. Clearly, $S'$ is a polar set-system that has degree at most $\delta$. Choose weights $w\in \bZ^{[d]\cup [\bar{d}]}_{\geq 0}$ satisfying \ref{it:h1} and \ref{it:h2}. Our minimal choice of $S'$ implies that \ref{it:h3} is also satisfied. It therefore follows from \Cref{thm:degree-dimension} that $S'$ has degree $d$, that is, $d\leq \delta$, thus finishing the proof.
\end{proof}

\section{Computational results}\label{section:computational}

In this section we present the computer-assisted part of the proof of
\cref{cor:deg_nine}.  The result shows that cuboids of cube-ideal strictly polar set-systems have the MFMC property whenever the set-system either has degree at most $9$, or is up-monotone and has dimension at most $10$.

Our computation reduces the task to verifying the unsatisfiability of a finite number of SAT instances indexed by `admissible' and `tightly admissible' weight functions. After describing such weight functions, we describe the SAT encoding, including its simplification in the up-monotone case, and then report results and running times.

\subsection{Admissible and tightly admissible weight functions}

We first define a property of weight functions that must be satisfied for a minimal counterexample to \Cref{cor:deg_nine}. 
Given $d, \tau \in \bZ_{\ge 0}$ and a weight function $w\in \bZ_{>0}^{[d]\cup [\bar{d}]}$, let $E_i$ denote the set of elements of weight $i$ for each positive integer $i$.
Let us call $w$ \emph{admissible} if $w_i+w_{\bar{i}} = \tau$ for $i \in [d]$ and there exists an integer $s \in \{\tau+1, \dots, d+1\}$ for which the inequality
\begin{align} \label{eq:technical}
	\frac{s}{s-\tau}-\frac{\tau-2}{k}+1 \le \sum_{i=1}^k \frac{(k+1-i)(s-\tau + i - 1)}{k s}|E_i|
\end{align}
holds for all $k \in \{1,\dots, \tau-1\}$. 
Furthermore, let us call $w$ \emph{tightly admissible} if $s=d+1$ is the only $s \in \{\tau+1, \dots, d+1\}$ for which \eqref{eq:technical} holds for all $k \in \{1, \dots, \tau-1\}$. We see that this property must be satisfied for an extremal minimal counterexample to \Cref{cor:deg_nine}. We computationally verified the following result using a SAT solver; see the next subsections for details.

\begin{prop} \label{prop:sat-verified}
	There exist no positive integers $d$ and $\tau$, set-system $S\subseteq \{0,1\}^d$, and weights $w\in \bZ^{[d]\cup [\bar{d}]}_{>0}$ such that each of the following holds:
	\begin{enumerate}[label=(c\arabic*)]
		\item $3 \le \tau \le d$, and either $d \le 9$, or $d=10$ and $S$ is up-monotone, \label{it:numerical-bounds}
		\item $w$ is admissible and $w_i \le w_{\bar{i}}$ for each $i \in [d]$, \label{it:weight-conditions}
		\item $S$ is cube-ideal, strictly polar, and $\0 \in \bar{S}$, \label{it:S-conditions}
		\item for $\cC\coloneqq \cuboid(S)$, we have $\tau(\cC, w) = \tau$, and if $w$ is tightly admissible, then the only covers of weight $\tau$ are $\{i, \bar{i}\}$ for $i \in [d]$. Moreover, for each $C\in \cC$, there exists a cover $B$ of $\cuboid(S)$ such that $w(B) \le |B\cap C|+\tau-2$. \label{it:cuboid-conditions}
	\end{enumerate}
\end{prop}

We show that this result implies \cref{cor:deg_nine}.

\begin{proof}[Proof of \Cref{cor:deg_nine}.]
	Assume to the contrary that the statement fails and let $S\subseteq \{0,1\}^d$ be a counterexample with minimum $d$. 
	Let $\cC \coloneqq \cuboid(S)$ and choose $w \in \bZ^{[d]\cup [\bar{d}]}_{\ge 0}$ such that $\nu(\cC, w) < \tau(\cC, w)=:\tau$ and $\nu(\cC, w') = \tau(\cC, w')$ for any $w'\in \bZ^{[d]\cup [\bar{d}]}$ with $w'\lneq w$.
	Then, \ref{it:h1}, \ref{it:h2}, and \ref{it:h3} are satisfied, hence \cref{prop:acl} implies that $w_i + w_{\bar{i}} = \tau$ and $1 \le w_i, w_{\bar{i}} \le \tau-1$ for each $i \in [d]$, $\tau \ge 3$, and each $C \in \cC$ has a mate $B$.
	By \cref{lem:zero-infeasible}, we may assume that $w_i \le w_{\bar{i}}$ for $i \in [d]$ and $\0 \in \bar{S}$.
	\cref{thm:degree-dimension} implies that $S$ has degree $d$, in particular, either $d \le 9$, or $d=10$ and $S$ is up-monotone.	
	\cref{mfmc-testing-cuboid} shows that $\tau \le d$.
	Finally, using the notation $E_i \coloneqq \{e \in [d] \cup [\bar{d}]: w_e = i\}$ for $i \in \{1,\dots, \tau-1\}$, 
	\Cref{cl:average-mate-small-weight} and \Cref{cl:average-set-small-weight} in the proof of \cref{mfmc-testing-ideal} show that for some integer $s \in \{\tau+1, \dots, \rank(M(\cC))\}$, the inequality \eqref{eq:technical} holds for each $k \in \{1,\ldots,\tau-1\}$.
	Since $s \le \rank(M(\cC)) \le d+1$, we get that $w$ is admissible.
	If $w$ is tightly admissible, then $s=\rank(M(\cC)) = d+1$, hence the only minimum weight covers are $\{i, \bar{i}\}$ for $i \in [d]$.
	These together show that \ref{it:numerical-bounds}, \ref{it:weight-conditions}, \ref{it:S-conditions}, and \ref{it:cuboid-conditions} are all satisfied, contradicting \cref{prop:sat-verified}.
\end{proof}

To verify \cref{prop:sat-verified}, for each fixed $d, \tau \in \bZ_{>0}$ and $w\in  \bZ_{>0}^{[d]\cup [\bar{d}]}$ satisfying \ref{it:numerical-bounds} and \ref{it:weight-conditions}, we encode the existence of $S\subseteq \{0,1\}^d$ satisfying \ref{it:S-conditions} and \ref{it:cuboid-conditions} as a SAT formula.
Since the conditions in \cref{prop:sat-verified} are invariant under permuting the coordinate pairs $\{i,\bar{i}\}$, it suffices in the computation to consider admissible weight functions satisfying $w_1 \le \cdots \le w_d$.
Before giving the encoding, we present some technical results on idealness that we will use.

\subsection{Preparations for the SAT encoding}

As a preparation for our computational result, we prove the following technical lemma. 
Its statement \ref{it:cube-ideal} gives a characterization of cube-idealness of strictly polar sets relying on Lehman's characterization of mni clutters and certain instances of the `width-length inequality'~\cite{lehman1979width}.
The latter describes ideal clutters $\cC$ on ground set $E$ as the ones satisfying \[\min \{w(C): C \in \cC\} \cdot \min \{\ell(B):  B\in b(\cC)\} \le w^\top \ell\ \text{ for all } w, \ell \in \bR_{\ge 0}^{E}.\]

\begin{lemma} \label{lem:cube-ideal-SAT}
	Let $P$ denote the set of triples of integers $(m, r, s)$ for which there exists an mni clutter $\cC'$ on $m$ elements with no intersecting minor satisfying $\min \{|C|: C \in \cC'\}=r$ and $\min\{|B|: B \in b(\cC')\} = s$. Then the following statements hold:
	\begin{enumerate}[label=(\alph*)]
		\item \label{it:P-general} $P\subseteq \left\{(m, r, s)\in \bZ^3_{>0}: 2 \le r = \left \lceil \frac{m}{s}\right \rceil, 3 \le s = \left \lceil \frac{m}{r}\right \rceil, r \nmid m, s \nmid m\right\}$.
		\item \label{it:P-small} $P \cap \left\{(m, r, s)\in \bZ_{> 0}^3: m \le 12\right\}$ equals $$\{(5,2,3), \ (7,2,4), \ (8,3,3), \ (9,2,5), \ (11,2,6), \ (11,3,4), \ (11,4,3)\}.$$
		\item \label{it:cube-ideal} Let $S\subseteq \{0,1\}^d$ be a strictly polar set-system. Then $S$ is cube-ideal if and only if the following holds: for any $(m, r, s) \in P$ with $m \le d$, any minor $S'$ of $S$ on $m$ elements, and any localization $\cC'$ of $S'$, $\cC'$ contains a member of size at most $r-1$ or a cover of size at most $s-1$.
	\end{enumerate}
\end{lemma}
\begin{proof}
	To show \ref{it:P-general}, let $(m, r, s) \in P$ and let $\cC'$ be a corresponding clutter.
	Since $\cC'$ has no intersecting minor, it is not a delta, so by Lehman's \cref{lehman}, $r, s \ge 2$ and there exist a labeling $C_1, \dots, C_m$ of the minimum-size members of $\cC'$ and a labeling $B_1, \dots, B_m$ of the minimum-size covers of $\cC'$ such that $|B_i \cap C_i| = rs-m+1\ge 2$ for each $i \in [m]$.
	Then, $2 \le rs-m+1 = |B_i\cap C_i| \le \min \{r, s\}$ implies $0 < rs-m < \min \{r, s\}$, hence $r = \left\lceil \frac{m}{s} \right\rceil$, $s = \left\lceil \frac{m}{r} \right\rceil$ and $m$ is not divisible by $r$ or $s$.
	We also have $s\ge 3$, since $s=2$ would imply that $r=\left\lceil \frac{m}{2} \right\rceil > \frac{m}{2}$, hence $\cC'$ would be intersecting.
	
	To show \ref{it:P-small}, we note first that for each triple listed in \ref{it:P-small}, there exists at least one mni clutter with these parameters that has no intersecting minor: $C^2_5$, $C^2_7$, $C^3_8$, $C^2_9$, $C^2_{11}$, $C^3_{11}$, and $C^4_{11}$, respectively, where $C^r_m$ is the clutter on $[m]$ whose members are the $m$ sets of $r$ cyclically consecutive elements~\cite{lutolf1998catalog}.
	To show the reverse containment, let $(m,r,s) \in P$ with $m \le 12$.
	Then, \ref{it:P-general} implies that $(m, r, s)$ is either a triple listed in \ref{it:P-small} or one of $(7,3,3)$, $(10,3,4)$, or $(10,4,3)$.
	Assume first that $(m, r, s) \in \{(7,3,3), (10,4,3)\}$.
	Then, $|B_i \cap C_i| = rs-m+1 = s=|B_i|$, hence $B_i \subseteq C_i$, thus no member of $\cC'$ is disjoint from $C_i$ for $i \in [m]$.
	For a clutter $\cC''$ obtained from $\cC'$ by deleting any $s-2$ elements, we have $\tau(\cC'') \ge \tau(\cC')-(s-2) = 2$, hence $\cC''$ contains two disjoint members as $\cC'$ has no intersecting minor.
	Since each member of $\cC'$ intersects each $C_i$, we have that each of the two disjoint members of $\cC'$ has size at least $r+1$, thus $2\cdot (r+1) \le m-s+2$, a contradiction for these parameters.
	Finally, if $(m,r,s) = (10,3,4)$, then the catalog of L\"{u}tolf and Margot~\cite{lutolf1998catalog} shows that the minimum cardinality members of $\cC'$ form a clutter isomorphic to $\mathcal{T}_{K_5}$: the clutter on the edge set of the complete graph $K_5$ consisting of the triangles. 
	Restricting this clutter to the edges of a $K_4$, we have an intersecting minor, a contradiction.
	
	Finally, to show \ref{it:cube-ideal}, assume first that $S$ is cube-ideal.
	Then, each minor $S'$ of $S$ is cube-ideal as well, hence each localization $\cC'$ of $S'$ is ideal.
	By the width-length inequality, $\min\{|C|: C \in \cC'\} \cdot \min \{|B|: B \in b(\cC')\} \le m$, where $m$ is the size of the ground set of $\cC'$.
	In particular, if $(m, r, s) \in P$, then $rs > m$, hence $\min\{|C|: C \in \cC'\} \le r-1$ or $\min\{|B|: B \in b(\cC')\} \le s-1$.
	To see the converse, assume that $S$ is not cube-ideal.
	Then, some localization of $S$ is non-ideal, and let $\cC'$ be an mni minor of that localization.
	Observe that $\cC'$ can also be obtained as the localization of a minor $S'$ of $S$.
	Let $m$ denote the size of the ground set of $\cC'$, $r\coloneqq \min \{|C|: C \in \cC'\}$ and $s\coloneqq \min \{|B|: B \in b(\cC')\}$.
	Then, $(m, r, s) \in P$ and the condition in \ref{it:cube-ideal} fails.
	This finishes the proof.  
\end{proof}

\subsection{The SAT encoding}

Let $d, \tau \in \bZ_{>0}$ and $w\in \bZ_{>0}^{[d]\cup [\bar{d}]}$ satisfy \ref{it:weight-conditions}, that is, $1 \le w_i \le w_{\bar{i}} \le \tau-1$ and $w_i +w_{\bar{i}} = \tau$ for $i \in [d]$.
For each $p \in \{0,1\}^d$, we introduce a binary variable $x_p$ indicating whether it belongs to the set-system.
In what follows, we describe a SAT formula in \emph{conjunctive normal form (CNF)}, given as a set of clauses involving $x_p,p\in \{0,1\}^d$ and other auxiliary variables such that for a set-system $S\subseteq \{0,1\}^d$, there is a satisfying truth assignment with $S= \{p \in \{0,1\}^d: x_p  \text{ is }\texttt{True}\}$ if and only if $S$ satisfies \ref{it:S-conditions} and \ref{it:cuboid-conditions}.

\paragraph{Strict polarity.}
By definition, $S$ is strictly polar if each of its restrictions is contained in a hyperplane of the form $\{x:x_i=a\}$, or contains two antipodal points.
To encode the former, we use auxiliary variables that will also be used later.
For each disjoint $I_0, I_1 \subseteq [d]$, we use the notation \[R_{I_0, I_1} \coloneqq \{p \in \{0,1\}^d: p_i = 0 \ (i \in I_0), p_i = 1\  (i \in I_1)\},\]
and we introduce a binary variable $y_{I_0, I_1}$ and add clauses such that \[y_{I_0, I_1} \text{being \texttt{True} corresponds to } S\cap R_{I_0,I_1} \ne \emptyset.\]
If $I_0 \cup I_1 = [d]$, we identify $y_{I_0, I_1}$ with $x_p$ where $p\in \{0,1\}^d$ is the unique point in $R_{I_0, I_1}$.
Otherwise, if $I_0 \cup I_1 \ne [d]$, then we fix a $j \in [d]\setminus (I_0 \cup I_1)$ and add three clauses enforcing \[y_{I_0, I_1} = y_{I_0\cup \{j\}, I_1} \vee y_{I_0, I_1\cup \{j\}}.\]
For disjoint $I_0, I_1 \subseteq [d]$, we formulate the polarity of the restriction $S\cap R_{I_0, I_1}$ as \[\left(\bigvee_{j \in [d]\setminus (I_0 \cup I_1)} (\bar{y_{I_0\cup \{j\}, I_1}} \vee \bar{y_{I_0, I_1\cup \{j\}}})\right) \vee \left(\bigvee_{p, q \text{ antipodal in $R_{I_0, I_1}$}} (x_p \wedge x_q)\right)\]
expanded into clauses using auxiliary variables. 
Adding this for each pair of disjoint $I_0, I_1 \subseteq [d]$ with $|I_0|+|I_1|\le d-3$, we encode exactly the strict polarity of $S$, as $(\leq 2)$-dimensional restrictions are automatically polar.

\paragraph{Cube-idealness.}
We enforce cube-idealness using \cref{lem:cube-ideal-SAT}, part~\ref{it:cube-ideal}.
For each $(m, r, s) \in P$ with $m \le d$, we add the constraint that each localization of each minor of $S$ on $m$ elements contains a member of size at most $r-1$ or a cover of size at most $s-1$.
Note that by \ref{it:P-small}, we only need to use 4 different triples of parameters for $d \le 10$.
For a given $(m, r, s)$, we iterate over all pairwise disjoint $I_0, I_1, J, K \subseteq [d]$ with $I_0 \cup I_1 \cup J \cup K = [d]$ and $|K| = m$, and over all $p \in \{0,1\}^K$.
For each such choice, let $S'$ be the minor of $S$ obtained by restricting the coordinates in $I_0$ to $0$, the coordinates in $I_1$ to $1$, and projecting away the coordinates in $J$. 
We add the constraint that the localization $\loc (S'; p)$ contains a member of size at most $r-1$ or a cover of size at most $s-1$.
We let $K_0 \coloneqq \{i \in K: p_i = 0\}$, $K_1 \coloneqq \{i \in K: p_i = 1\}$.
Observe that for $D\subseteq K$, the condition $S\cap R_{I_0 \cup (K_0 \cap D), I_1 \cup (K_1 \cap D)} \ne\emptyset$ is equivalent to $\loc(S'; p)$ having a member contained in $K\setminus D$.
Similarly, for $B\subseteq K$, the condition $S \cap R_{I_0 \cup (K_0 \cap B), I_1 \cup (K_1 \cap B)} = \emptyset$ is equivalent to $B$ being a cover of $\loc(S'; p)$.
Therefore, the clause
\[\left(\bigvee_{D \subseteq K:\ |D|=m-r+1} y_{I_0 \cup (K_0 \cap D),\ I_1 \cup (K_1 \cap D)}\right) \vee \left(\bigvee_{B \subseteq K:\ |B|= s-1} \bar{y_{I_0 \cup (K_0 \cap B), I_1 \cup (K_1 \cap B)}}\right)\]
adds the condition that $\loc(S'; p)$ contains a member of size at most $r-1$ or a cover of size at most $s-1$.

\paragraph{Constraints on covers.}
For the constraints involving covers of $\cC\coloneqq \cuboid(S)$, observe that a set $B\subseteq [d] \cup [\bar{d}]$ with $|B\cap \{i,\bar{i}\}| \le 1$ for all $i \in [d]$ is a cover of $\cC$ if and only if $S\cap R_{I_0, I_1} = \emptyset$ for $I_0 = B\cap [d]$ and $I_1 = \{i \in [d]: \bar{i}\in B\}$.
We forbid covers of weight at most $\tau-1$ by iterating over each pair of disjoint $I_0, I_1 \subseteq [d]$ with $\sum_{i\in I_0} w_i + \sum_{i \in I_1} w_{\bar{i}} \le \tau-1$ and setting $y_{I_0, I_1}$ to \texttt{True}.
This guarantees $\tau(\cC, w) = \tau$, since $\tau(\cC, w) \le \tau$ holds due to the weight-$\tau$ covers $\{i, \bar{i}\}$ for $i \in [d]$.
If $w$ is tightly admissible, as guaranteed by \ref{it:cuboid-conditions}, we similarly forbid all covers $B$ satisfying $w(B) = \tau$ and $|B\cap \{i,\bar{i}\}| \le 1$ for each $i \in [d]$.
For the condition on mates appearing in \ref{it:cuboid-conditions}, note that it is equivalent to requiring for each $C\in \cC$ the existence of a (not necessarily minimal) cover $B$ with $w(B) \le |B\cap C| + \tau-2$ and $|B\cap \{i,\bar{i}\}|\le 1$ for $i \in [d]$.
To add this, we iterate over $p \in \{0,1\}^d$, and add the constraint 
\begin{align} \label{eq:mates}
	\bar{x_p} \vee \bigvee_{\substack{I_0, I_1 \subseteq [d], \ I_0 \cap I_1 = \emptyset, \\ \sum_{i \in I_0} w_i + \sum_{i \in I_1} w_{\bar{i}} \le |\{i \in I_0: p_i = 1\}| + |\{i \in I_1: p_i = 0\}| + \tau-2}} \bar{y_{I_0, I_1}}.
\end{align}
This finishes the description of all conditions of \ref{it:S-conditions} and \ref{it:cuboid-conditions}, except for $\0 \in \bar{S}$ which we add by setting $x_{\0}$ to \texttt{False}.

\paragraph{Symmetry breaking.}
Finally, we add some clauses that preserve the satisfiability of the formula while reducing the search space.
Let $\Gamma$ be the subgroup of permutations $\sigma$ of $[d]$ with $w_{\sigma(i)} = w_i$ for each $i \in [d]$.
For a permutation $\sigma \in \Gamma$ and a point $p \in \{0,1\}^d$, let $\sigma(p) \coloneqq (p_{\sigma^{-1}(1)}, \dots, p_{\sigma^{-1}(d)})$.
Observe that for each $\sigma \in \Gamma$ and $S\subseteq \{0,1\}^d$, the set-system $\{\sigma(p): p \in S\}$ satisfies \ref{it:S-conditions} and \ref{it:cuboid-conditions} whenever $S$ does.
Fix an enumeration $p_1, \dots, p_{2^d}$ of $\{0,1\}^d$.
Adding the constraint
\begin{align} \label{eq:lex}
(x_{p_1}, \dots, x_{p_{2^d}}) \ge_{\mathrm{lex}} (x_{\sigma(p_1)}, \dots, x_{\sigma(p_{2^d})})
\end{align}
for each $\sigma \in \Gamma$ would not change the satisfiability of the formula since any lexicographically maximal satisfying assignment of the original formula also satisfies these constraints. Here, $x_p$ is a binary variable for each point $p\in \{0,1\}^d$, where $x_p=1$ if $x_p$ is \texttt{True}, and $x_p=0$ if $x_p$ is \texttt{False}. As adding the constraint \eqref{eq:lex} for each $\sigma \in \Gamma$ would be impractical if $|\Gamma|$ is large, we add it only for the transpositions in $\Gamma$.
We encode \eqref{eq:lex} with clauses in the standard way using auxiliary equal-prefix variables~\cite{crawford1996symmetry}.

\paragraph{Up-monotone set-systems.}
The SAT encoding described thus far is the generic one we used for the computations with $d\leq 9$. For the up-monotone case in dimension $10$, we used additional clauses and simplifications to speed up the computation. 

The up-monotonicity of $S$ can be simply enforced by adding $\bar{x_p} \vee x_q$ for each $p, q \in \{0,1\}^d$ with $p \le q$ where $p$ and $q$ differ by exactly one coordinate.
When restricting the search space to up-monotone set-systems, we may omit some of the clauses for cube-idealness and covers.
First, for the cube-idealness clauses it suffices to consider only the cases with $I_1=\emptyset$ and $p=\0$, since all other localizations reduce to these after deleting free elements.
Namely, if $I_0,I_1,J,K$ is a partition of $[d]$, $p\in \{0,1\}^K$, $P = \{i\in K : p_i=1\}$, and $S'\subseteq \{0,1\}^K$ is obtained from $S$ by $0$-restricting $I_0$, $1$-restricting $I_1$, and projecting $J$, then $\loc(S';p)$ can be obtained from $\loc(S;\0)\setminus I_0 /(I_1\cup J\cup P)$ by adding the elements of $P$ as free elements (see also \cite[Remark~4.6]{abdi2020cuboids}).
Second, for adding the condition on mates, in \eqref{eq:mates} it suffices to consider $I_1 = \emptyset$, since for disjoint $I_0, I_1 \subseteq [d]$, $I_0 \cup \{\bar{i}: i \in I_1\}$ is a cover of $\cuboid(S)$ if and only if $I_0$ is.

\subsection{Results and running times}

Using the above encoding, we ran the SAT solver CaDiCaL~\cite{cadical}
on every instance with $d\leq 9$ and every up-monotone instance
with $d=10$. In each case CaDiCaL returned UNSAT and produced a DRAT
certificate of unsatisfiability. These certificates were independently
verified using \texttt{drat-trim}~\cite{drattrim}, and they form the
computer-assisted proof of \cref{prop:sat-verified}.

The implementation, generated CNF instances, DRAT proofs, and solver and verifier metadata are archived on Zenodo~\cite{mfmc_data}, while the implementation is also available on GitHub\footnote{\url{https://github.com/tamas-schwarcz/mfmc}}. 
The running times in \cref{table:times} are sums of the per-instance CaDiCaL wall-clock times recorded in the solver metadata. Thus they measure aggregate solver time, not elapsed batch time.
The computations were run with four parallel worker processes on a personal
computer equipped with an 11th Gen Intel Core i7-1165G7 processor at 2.80 GHz (4 physical cores, 8 threads) and 32 GB of RAM.

\begin{table}[htbp] \centering
\begin{tabular}{llrrrrrr}
	\toprule
	$d$ & set-system & $\max\tau$ & \#adm. & \#tight & $\max$ vars & $\max$ clauses & time (s)\\
	\midrule
	5 & general & 3 & 1 & 1 & 819 & 2\,518 & 0.01 \\
	6 & general &  4 & 4 & 3 & 3\,033 & 9\,859 & 0.18 \\
	7 & general &  5 & 10 & 5 & 11\,211 & 40\,543 & 10.04 \\
	8 & general &  6 & 18 & 3 & 41\,761 & 180\,116 & 932.49 \\
	9 & general &  6 & 58 & 18 & 157\,411 & 859\,312 & 319\,423.09 \\
	10 & up-mon. & 7 & 130 & 50 & 600\,745 & 1\,856\,940 & 453\,409.52 \\
	\bottomrule
\end{tabular} 
	\caption{Summary of the SAT computations by dimension.
		The rows for $d=5,\ldots,9$ concern the general setup, while $d=10$ is only for up-monotone set-systems.		
		For each dimension $d$, the table gives the largest value of $\tau$ for which an admissible weight function exists, the number of admissible weight functions
		$w \in \bZ_{>0}^{[d] \cup [\bar{d}]}$ with $w_1\leq\cdots\leq w_d$, the number of tightly
		admissible weight functions among them, the maximum numbers of variables and
		clauses among the corresponding CNF formulas, and the sum of the per-instance CaDiCaL wall-clock times recorded for that dimension.} \label{table:times}
\end{table}

As \cref{table:times} shows, in the dimensions considered here the admissibility condition rules out many values of $\tau$ allowed by the a priori bound in \cref{mfmc-testing-cuboid}.
The tightly admissible instances were typically much easier to solve. For
example, for $d=9$ they accounted for less than $0.1\%$ of the aggregate
per-instance CaDiCaL time.
The slowest instance among the ones reported in \cref{table:times} was the instance with $d=9$, $\tau=3$, and $(w_1, \dots, w_{9}) = (1,1,1,1,1,1,2,2,2)$ and it took roughly 9 hours to solve.

The restriction to up-monotone instances speeds up the solver significantly; e.g. for $d=9$, adding up-monotonicity sped up the computation by a $178$ factor (the solver time reduced from $319\,423.09\,\mathrm{s}$ down to $1790.62\,\mathrm{s}$). 
Therefore, with our current setup, we did not complete the runs for general instances with $d=10$.
Nevertheless, we expect that stronger symmetry-breaking clauses and additional computational resources would allow the general bound $d\leq 9$ and the up-monotone bound $d\leq 10$ to be pushed further.

\section{Future research directions}\label{section:conclusion}

In this paper, we verified the replication conjecture for all cuboids of degree at most $9$, and all clutters over at most $10$ elements. Our results were built on two pillars. One was to upper bound the weighted covering number for the search directions needed to test the MFMC property, and another to upper bound the dimension of minimally non-MFMC cuboids of bounded degree. We conclude with some thoughts on how to strengthen both pillars.\medskip

The first future direction is to improve the bound on $\tau(\cC,w)$ in \Cref{mfmc-testing-ideal}. The guiding light here is the $\tau=2$ conjecture, which can be reformulated as follows.

\begin{conj}\label{CN:testing-mfmc}
Let $\cC$ be an ideal clutter over $n$ elements. If $\cC$ is not MFMC, then there exists $w\in \bZ_{\geq 0}^n$ such that every element belongs to a minimum $w$-weight cover, $\tau(\cC,w)=2$ and $\nu(\cC,w)=1$.
\end{conj}

An already non-trivial improvement of \Cref{mfmc-testing-ideal} would be to prove an upper bound on $\tau(\cC,w)$ of $cn$ for a universal constant $c\in (0,1)$.\medskip 

Another tightly linked approach is to upper bound the covering number of an ideal mnp clutter. We have the following starting point. 

\begin{thm}\label{tau=2-relaxation}
Let $\cC$ be an ideal mnp clutter over $n$ elements. Then \begin{enumerate}[label=(\alph*)]
\item \label{part:relaxation1} $\tau(\cC,\1)\leq \frac{n}{3}$, and 
\item \label{part:relaxation2} $\tau(\cC,\1)\leq r\left(1-\frac{1+\sqrt{4n+1}}{2n}\right)$, where $r=\rank(M(\cC))$. 
\end{enumerate}
\end{thm}
\begin{proof}
Let $E$ be the ground set of $\cC$ and let $\tau\coloneqq \tau(\cC,\1)$.
{\bf \ref{part:relaxation1}}  
By \Cref{minimal-non-MFMC-weights} applied for $\bar{w}=\1$, we obtain that every element belongs to a minimum cover, and every set $C$ in $\cC$ has a mate, i.e., a minimal cover $B$ such that $|B|\leq \tau-2+|B\cap C|$. As $\cC$ is ideal, there exists an optimal solution $y\in \bR^{\cC}_{\geq 0}$ to $\nu_\star(\cC,\1)$ of value $\tau$. 
The key idea is that if $y_C>0$, then $|C|\geq 3$. 
Such a $C$ intersects every minimum cover exactly once by complementary slackness. Thus, for $B$ a mate of $C$, given that $|B\cap C|\geq |B|-\tau+2\geq 2$, it follows that $B$ is non-minimum, so $|B|\geq \tau+1$, implying in turn that $|B\cap C|\geq 3$, so $|C|\geq 3$. Subsequently, $$
n\geq \sum_{e\in E}\sum_{C\in \supp(y),e\in C}y_C = \sum_{C\in \supp(y)} |C|y_C \geq 3\cdot \1^\top y = 3\tau,
$$ thus proving the claimed inequality.
{\bf \ref{part:relaxation2}} 
is obtained by adapting the proof of \Cref{mfmc-testing-ideal} for ideal mnp clutters by setting $\hat{w}=\1$. More specifically, one can improve \Cref{cl:average-mate-small-weight} and \Cref{cl:average-set-small-weight} to obtain that $\frac{(s-\tau)n}{s}\geq \Phi_1\geq \frac{s}{s-\tau}+1$, 
thus giving $d\geq s(1+\sqrt{4n+1})/2n$, 
eventually leading to the claimed upper bound on $\tau$.
\end{proof}

It would be quite interesting to provide a meaningful improvement of this theorem. In particular, any upper bound on $\tau(\cC,\1)$ of the form $n^{0.5-\varepsilon},\, \varepsilon\in (0,0.5)$ in \Cref{tau=2-relaxation} would immediately improve the upper bound on $\tau(\cC,w)$ in \Cref{mfmc-testing-ideal} to $n^{1-2\varepsilon}$ \`{a} la the proof of \Cref{tau=2->replicon}. \medskip

The second future direction is to improve 
\Cref{non-MFMC-cuboid}. 
To contextualize this, recall that idealness is a local property. This means that if $\cuboid(S)$ is non-ideal, then one of the localizations of $S$ is a non-ideal clutter. Thus, a non-ideal cuboid of degree $\delta$ has a non-ideal clutter minor over at most $\delta$ elements.

On the other hand, the MFMC property is not a local property. For instance, the cuboid of $\{000,110,101,011\}$ is not MFMC while all its localizations are. Given that strict polarity makes the packing property local, and given the replication conjecture, it is only natural to conjecture the following.

\begin{conj}\label{conj:MFMC-local}
Let $S$ be a strictly polar set-system. Then $\cuboid(S)$ is MFMC if, and only if, every localization of $S$ is MFMC.
\end{conj}

If true, this would imply that if the cuboid of a strictly polar set-system of degree $\delta$ is non-MFMC, then it has a non-MFMC clutter minor over at most $\delta$ elements, thus improving \Cref{non-MFMC-cuboid} by a multiplicative factor of $2$, as the latter only yields a non-MFMC minor over at most $2\delta$ elements.

\Cref{conj:MFMC-local} has the following attractive reformulation. A clutter is \emph{minimally non-MFMC} if it is non-MFMC but every proper minor is MFMC.
		
		\begin{conj}
			Let $\cC$ be an ideal minimally non-MFMC clutter. If there exist distinct elements $u,v$ such that $|C\cap \{u,v\}|=1$ for all $C\in \cC$, then $2=\tau(\cC,\1)>\nu(\cC,\1)$.
		\end{conj}

\section*{Acknowledgements}

This work was supported by EPSRC grant EP/X030989/1. 
This work was initiated when AA was a Visiting Professor hosted generously by the G-SCOP Laboratory, Université Grenoble Alpes in Fall 2024. 
We would like to thank Alex Black, G\'{e}rard Cornu\'{e}jols, Andr\'{a}s Seb\H{o}, and Olha Silina for early discussions about this work.

\paragraph{AI Statement.}
The authors used large language models --- specifically ChatGPT (OpenAI) and Claude (Anthropic) --- to prepare and refine supporting scripts for reproducibility and data-management, and to polish the instance generator. The authors thoroughly reviewed and tested all AI-generated material and take full responsibility for the final content.

\paragraph{Data Availability Statement.} The computational artifact supporting the computer-assisted results in this paper is available on Zenodo~\cite{mfmc_data} at  \url{https://doi.org/10.5281/zenodo.20590528}. The archive contains the Python implementation, generated CNF instances, DRAT proofs, solver and verifier metadata, manifests, checksums, and instructions for reproducing the computations. The implementation is also available on GitHub at \url{https://github.com/tamas-schwarcz/mfmc}.

\paragraph{Conflicts of interests/competing interests statement.} The authors have no conflicts of interest nor competing interests to declare that are relevant to the content of this article.
	
{\small 
\bibliographystyle{alpha}
\bibliography{main.bib}
	}
\end{document}